\documentclass{amsart}[12pt]
\usepackage{amsmath, amsthm, amscd, amsfonts,}
\usepackage{graphicx,float}

\usepackage{amssymb}

\usepackage{pb-diagram}
\usepackage{lamsarrow}
\usepackage{pb-lams}

\newcommand{\PP}{{\mathbb{P}}}
\newcommand{\QQ}{{\mathbb{Q}}}

\newcommand{\ga}{\alpha}

\newcommand{\gk}{\kappa}

\newcommand{\gn}{\nu}

\newcommand{\gt}{\tau}

\newcommand{\ZF}{\hbox{ZF}}

\newcommand{\func }{\mathord{:}}

\newcommand{\restricted}{\mathord{\restriction}}
\newcommand{\upto}{\mathord{<}}

\newcommand{\Card}{\ensuremath{\text{Card}}}
\newcommand{\ordered}[1]{\ensuremath{\langle #1 \rangle}}

\newcommand{\ordof}[2]{\ensuremath{\ordered{ #1 \mid #2 }}}

\DeclareMathOperator{\len}{l}

\DeclareMathOperator{\crit}{crit}
\DeclareMathOperator{\dom}{dom}

\DeclareMathOperator{\Col}{Col}
\DeclareMathOperator{\Add}{Add}

\DeclareMathOperator{\supp}{supp}
\DeclareMathOperator{\Ult}{Ult}
\DeclareMathOperator{\limdir}{lim\ dir}

\DeclareMathOperator{\mc}{mc}

\newcommand{\Es}{{\ensuremath{\bar{E}}\/}}

\newcommand{\PE}{{\ensuremath{P_{\Es}\/}}}

\newcommand{\VS}{V^*}
\newcommand{\MS}{{M^*}}

\newcommand{\NS}{{N^*}}

\newcommand{\MSt}{{M^*_\gt}}
\newcommand{\Mt}{{M_\gt}}

\newcommand{\ME}{{M_{\Es}}}

\newcommand{\MSE}{{M^*_{\Es}}}

\def\MPB{{\mathbb{P}}}

\newtheorem{theorem}{Theorem}[section]
\newtheorem{lemma}[theorem]{Lemma}

\newtheorem{definition}[theorem]{Definition}

\newtheorem{remark}[theorem]{Remark}
\newtheorem{claim}[theorem]{Claim}
\numberwithin{equation}{section}

\usepackage{pb-diagram}

\def\rmark{\mbox{$\rm\bf\rule{0.06em}{1.45ex}\kern-0.05em R$}}
\def\pmark{\mbox{$\rm\bf\rule{0.06em}{1.45ex}\kern-0.05em P$}}
\def\nmark{\mbox{$\rm\bf\rule{0.06em}{1.45ex}\kern-0.05em N$}}
\def\vdash{\mbox{$\rm\| \kern-0.13em -$}}
\newcommand{\lusim}[1]{\smash{\underset{\raisebox{1.2pt}[0cm][0cm]{$\sim$}}
{{#1}}}}

\begin{document}

\title[Killing the GCH everyhwere with a single real]{Killing the GCH everyhwere with a single real}

\author[Sy D. Friedman and M. Golshani]{Sy-David Friedman and Mohammad
  Golshani}

\thanks{The first author would like to thank the Austrian Science Fund (FWF)
for its support through research project P 22430-N13.}

\thanks{The second author's research was in part supported by a grant from IPM (No. 91030417). He also wishes to thank the Austrian Science Fund (FWF)
for its support through research project P 223316-N13.} \maketitle

%{ \\ Department of Mathematics\\  Shahid Bahonar University of Kerman, Kerman, Iran}

%\subjclass[2010]{03E35}

%\keywords{ Multiplication mudule, prime submodule, Strongly prime submodule, valuation,
%fractional submodule, pseudo-valuation module}

\begin{abstract}
Shelah-Woodin \cite{shelah-woodin} investigate the possibility of
violating instances of $GCH$ through the addition of a single
real. In particular they show that it is possible to obtain a
failure of $CH$ by adding a single real to a model of $GCH$,
preserving cofinalities. In this article we strengthen their
result by showing that it is possible to violate $GCH$ at all
infinite cardinals by adding a single real to a model of $GCH.$
Our assumption is the existence of an $H(\kappa^{+3})$-strong
cardinal; by work of Gitik and Mitchell \cite{gitik-mitchell} it
is known that more than an $H(\kappa^{++})$-strong cardinal is
required.

\end{abstract}
\maketitle

\section{Introduction}

Shelah-Woodin \cite{shelah-woodin} investigate the possibility of
violating instances of GCH through the addition of a single real. In particular they show that it is possible to obtain a
failure of $CH$ by adding a single real to a model of $GCH$,
preserving cofinalities. In
this article we bring this work to its natural conclusion by showing
that it is possible to violate GCH at all infinite cardinals by adding
a single real to a model of GCH.

\begin{theorem}
Assume the consistency of an $H(\kappa^{+3})$-strong cardinal
$\kappa$. Then there exists a pair $(W,V)$ of models of $ZFC$ such that:

$(a)$ $W$ and $V$ have the same cardinals,

$(b)$ $GCH$ holds in $W,$

$(c)$ $V = W[R]$ for some real $R$,

$(d)$ $GCH$ fails at all infinite cardinals in $V.$
\end{theorem}

To achieve the result we add a generic for a Prikry product, code it
by a real preserving $H(\kappa^{+3})$-strength and
then finish the proof by quoting a modifed version of a result of Merimovich
\cite{merimovich}.

We also show that assuming the existence of a proper class of measurable cardinals, it is possible to force Easton's theorem by adding a single real. More precisely:

\begin{theorem}
Let $M$ be a model of $ZFC+GCH+$ there exists a proper class of measurable cardinals. In $M$ let $F:REG \longrightarrow CARD$ be an Easton function, i.e a definable class function such that

$\hspace{1cm}\bullet$ $\kappa \leq \lambda \longrightarrow F(\kappa) \leq F(\lambda)$, and

$\hspace{1cm}\bullet$ $cf(F(\kappa)) > \kappa$.

Then there exists a pair $(W,V)$ of cardinal preserving extensions of $M$ such that

$(a)$ $W \models \ulcorner GCH \urcorner $,

$(b)$ $V = W[R]$ for some real $R$,

$(c)$ $V \models \ulcorner \forall \kappa \in REG, 2^{\kappa} \geq F(\kappa) \urcorner$.
\end{theorem}

The reason that in $(c)$ we do not require equality is that it might be possible that $F(\kappa)$ changes its cofinality in $V$ to $\omega,$ and then clearly $2^{\kappa} \neq F(\kappa)$ in $V$. To achieve the result we define a class forcing version of the Prikry product, code its generic by a real using Jensen's coding and then finish the proof by applying Easton's theorem.

\section{Proof of Theorem 1.1}

\subsection{Prikry products}

Assume GCH and suppose that $S$ is a set of measurable cardinals which is
\emph{discrete}, i.e., contains none of its limit points. Fix normal
measures $U_\alpha$ on $\alpha$ for $\alpha$ in $S$. Then $\mathbb{P}$$_S$
denotes the Prikry product of the forcings $\mathbb{P}_\alpha$,
$\alpha\in S$, where $\mathbb{P}_\alpha$ is the Prikry forcing associated with
the measure $U_\alpha$.
Thus
\begin{center}
$\PP_S = \{ \langle (s_{\alpha}, A_{\alpha}): \alpha \in S \rangle\in \prod_{\alpha \in S} \PP_{\alpha} : s_{\alpha}=\emptyset$ for all but finitely many $\alpha \in S  \}.$
\end{center}
For two conditions $p=\langle (s_{\alpha}, A_{\alpha}): \alpha \in S
\rangle$ and $q=\langle (t_{\alpha}, B_{\alpha}): \alpha \in S
\rangle$ in $\PP_S$ we define $p\leq q$ ($p$ is stronger than $q$) if
$(s_{\alpha}, A_{\alpha})\leq (t_{\alpha}, B_{\alpha})$ in
$\PP_{\alpha}$ for all $\alpha \in S.$ We also define the auxiliary
relation $p\leq^* q$ ($p$ is a direct or a Prikry extension of $q$) if
$p\leq q$ and $s_{\alpha}=t_{\alpha}$ for all $\alpha \in S.$
\footnote{Thus $\PP_S$ is forcing equivalent to the Magidor iteration
  of the Prikry forcings $\PP_{\alpha}, \alpha \in S.$}

A $\mathbb{P}$$_S$-generic is uniquely determined by a sequence
$\langle x_\alpha:\alpha\in S \rangle$, where each $x_\alpha$ is an $\omega$-sequence
cofinal in $\alpha$.
With a slight abuse of terminology, we say that
$\langle x_\alpha:\alpha\in S \rangle$ is \emph{$\PP_S$-generic}.

\begin{lemma} (Fuchs \cite{fuchs}, Magidor \cite{magidor}) \label{fuchs}
Suppose that $\langle x_\alpha:\alpha\in S \rangle$ is $\PP_S$-generic over $V$.\\
(a) $V$ and $V[\langle x_\alpha:\alpha\in S \rangle]$ have the same cardinals.\\
(b) The sequence $\langle x_\alpha:\alpha\in S \rangle$ obeys the
following ``geometric property'': If
$\langle X_\alpha:\alpha\in S \rangle$ belongs to $V$ and
$X_\alpha\in U_\alpha$ for each $\alpha\in S$,
then $\bigcup_{\alpha\in S} x_\alpha\setminus X_\alpha$ is finite.\\
(c) Conversely, suppose that $\langle y_\alpha:\alpha\in S \rangle$ is a sequence
(in any outer model of $V$)
satisfying the geometric property stated above. Then
$\langle y_\alpha: \alpha\in S \rangle$ is $\PP_S$-generic over $V$.\\
(d) Suppose $\alpha \in S$, $p \in \PP_S $ and $ \langle \Phi_{\gamma}: \gamma < \eta \rangle$ is a sequence of statements of the forcing language for $\PP_S$ where $\eta < \alpha.$ Then there exists $q \leq^* p$ such that $q \upharpoonright \alpha = p \upharpoonright \alpha$ and for each $\gamma < \eta$ if $r\leq q$ and $r$ decides $\Phi_{\gamma},$ then $(r \upharpoonright \alpha )\cup (q \upharpoonright [\alpha, \kappa))$ (where $\kappa=\sup(S)$) decides $\Phi_{\gamma}$ in the same way.
\end{lemma}

\begin{theorem}
Suppose that $\kappa$ is $H(\kappa^{+3})$-strong and $S$ is a discrete set
of measurable cardinals less than $\kappa$. Then after forcing with
$\mathbb{P}$$_S$, $\kappa$ remains $H(\kappa^{+3})$-strong.
\end{theorem}

\begin{proof}
Suppose that $j: V \rightarrow M \supseteq H(\kappa^{+3}),$
$crit(j)=\kappa$ is an elementary embedding witnessing the  $H(\kappa^{+3})-$strength of $\kappa.$ We can assume that $j$ is derived from an extender $E= \langle E_{a}: a \in [\kappa^{+3}]^{< \omega} \rangle$. Then for each $a \in [\kappa^{+3}]^{< \omega} , E_{a}$ is a $\kappa-$complete ultrafilter on $[\kappa]^{|a|}$ and if $j_{a}: V \rightarrow M_{a} \cong Ult(V, E_{a})$ is the corresponding elementary embedding then for all $B \subseteq [\kappa]^{|a|},$ we have
 $B \in E_{a} \Leftrightarrow a \in j_{a}(B).$
We also have an embedding $k_{a}: M_{a} \rightarrow M$ such that $k_{a}\circ j_{a}=j.$

 We show that $\kappa$ remains $H(\kappa^{+3})-$strong in the generic extension by $\mathbb{P}$$_{S}$. The proof uses ideas from \cite{gitik-shelah} and \cite{magidor}. Let $G$ be $\mathbb{P}$$_{S}-$generic over $V$.
Also let $\delta=min(j(S)- \kappa) > \kappa.$

Working in $V[G],$ we define for each $a \in [\kappa^{+3}]^{< \omega_1}, E_{a}^{*}$ as follows: Let $\xi=o.t(a),$ and let $\dot{a}$ be a $\mathbb{P}$$_{S}-$name for $a$ such that
\begin{center}
 $\vdash \ulcorner \dot{a} \subseteq \kappa^{+3}$ and $o.t(\dot{a})=\xi \urcorner$
\end{center}
For $p \in \mathbb{P}$$_{S}$ define
$p \vdash \ulcorner \dot{B} \in \dot{E}_{a}^{*} \urcorner $ iff
\\
(1) $p \vdash \ulcorner \dot{B} \subseteq [\kappa]^{\xi} \urcorner,$
\\
(2) there exists $q \leq^{*} j(p)$ in $j(\mathbb{P}$$_{S})$ such that $q\upharpoonright\delta=j(p)\upharpoonright\delta=p,$ and
$q \vdash^{M} \ulcorner \dot{a} \in j(\dot{B})  \urcorner.$

Let $E_{a}^{*}=\dot{E}_{a}^{*}[G].$
It is easily seen that the above definition is well-defined.

\begin{lemma}
$(a)$ $E_{a}^{*}$ is a $\kappa-$complete non-principal ultrafilter on $[\kappa]^{\xi},$

$(b)$ If $a \in V$ is finite, then $E_{a}^{*}$ extends $E_{a}$,
\end{lemma}
\begin{proof}
$(a)$ We just prove that $E_{a}^{*}$ is $\kappa-$complete.  Suppose that $p \in \mathbb{P}$$_{S}$ and $p \vdash \ulcorner      [\kappa]^{\xi}= \bigcup \{\dot{B}_{\gamma}: \gamma < \eta \}  \urcorner$ where $\eta < \kappa.$ Then $j(p) \vdash^{M} \ulcorner      [j(\kappa)]^{\xi}= \bigcup \{j(\dot{B}_{\gamma}): \gamma < \eta \}  \urcorner$.

Working in $M$ consider $\delta, j(p)$ and the sequence $( \Phi_{\gamma}: \gamma < \eta )$ of sentences where for each $\gamma < \eta, \Phi_{\gamma}$ is ``$ \dot{a} \in j(\dot{B}_{\gamma})$'' It then follows from Lemma 2.1.(d) that there is $q \leq^{*} j(p)$ in $j(\mathbb{P}$$_{S})$ such that for each $\gamma < \eta$
\begin{itemize}
\item $q\upharpoonright \delta=j(p)\upharpoonright \delta=p,$ \item if $r \leq q$ and $r$ decides $\Phi_{\gamma},$ then $(r\upharpoonright\delta) \cup (q\upharpoonright [\delta,j(\kappa))$ decides $\Phi_{\gamma}$ in the same way.
\end{itemize}
Now $q \vdash^{M} \ulcorner \dot{a} \in [j(\kappa)]^{\xi}=\bigcup \{j(\dot{B}_{\gamma}): \gamma < \eta \}  \urcorner$ and hence we can find $r \leq q$ and $\gamma < \eta$ such that $r \vdash \ulcorner \Phi_{\gamma} \urcorner.$ Let $t=(r \upharpoonright \delta) \cup (q\upharpoonright [\delta,j(\kappa)).$ It is now easy to show that $t\upharpoonright \delta \leq p$ and $t\upharpoonright \delta \vdash \ulcorner   \dot{B}_{\gamma} \in \dot{E_{a}^{*}}\urcorner.$ This completes the proof of the $\kappa-$completeness of $E_{a}^{*}.$

$(b)$ Suppose $a \in V$ is finite. Let $B \in E_{a}$ and $p \in \mathbb{P}$$_{S}.$ We show that $p \vdash \ulcorner B \in \dot{E}_{a}^{*}  \urcorner.$ Let $q=j(p).$ Then $q$ has the required properties in the definition above which gives the result.
\end{proof}

In $V[G],$ for each $a \in [\kappa^{+3}]^{< \omega_1}$ let $j_{a}^{*}: V[G] \rightarrow M_{a}^{*} \simeq Ult(V[G], E_{a}^{*})$ be the corresponding elementary embedding. Also for $a \subseteq b$ let $k_{a,b}: M_{a}^{*} \rightarrow M_{b}^{*}$ be the natural induced elementary embedding. Let
\begin{center}
$\langle M^{*}, \langle k_{a}^{*}: a \in [\kappa^{+3}]^{< \omega_1} \rangle  \rangle=dirlim \langle \langle M_{a}^{*}:a \in [\kappa^{+3}]^{< \omega_1} \rangle, \langle k_{a,b}^{*}: a \subseteq b \rangle \rangle.  $
\end{center}
Also let $j^{*}: V[G] \rightarrow M^{*}$  be the induced embedding.
\begin{lemma}
$M^{*}$ is well-founded
\end{lemma}
\begin{proof}
Suppose not. Then there is a sequence $(m_i: i< \omega )$ of elements of $M^{*}$ such that
\begin{center}
$... \in^{*} m_2 \in^{*} m_1 \in^{*} m_0$
\end{center}
where $\in^{*}=\in_{M^{*}}.$ For each $i < \omega$ choose $a_i$ and $f_i$ such that $m_i=k_{a_i}^{*}([f_i]_{E_{a_i}^{*}}).$ Let $a= \bigcup \{ a_i:i < \omega \}.$ Then $a \in [\kappa^{+3}]^{< \omega_1}$ and for some $g_i, m_i=k_{a}^{*}([g_i]_{E_{a}^{*}}).$ It then follows from the elementarity of $k_{a}^{*}$ that
\begin{center}
$...\in [g_2]_{E_{a}^{*}} \in [g_1]_{E_{a}^{*}} \in [g_0]_{E_{a}^{*}}.$
\end{center}

This is in contradiction with Lemma 2.3 which implies $M_{a}^{*}$ is well-founded. Thus $M^{*}$ is well-founded and the lemma follows.
\end{proof}

If now we restrict ourself to $E_{a}^{*}$ for finite $a$, then the smaller direct limit embeds into the full direct limit and is therefore well-founded. From now on, let $M^*$ denote the smaller direct limit; accordingly each $E_{a}^{*}$ is now given by the usual extender definition and $j^*$ is the ultrapower embedding.

Note that $j^{*}: V[G] \rightarrow M^{*}$ is an elementary embedding with critical point $\kappa.$ We show that it is an $H(\kappa^{+3})-$strong embedding. For this it suffices to show that $H(\kappa^{+3})^{V[G]} \subseteq M^{*}.$ But since  $H(\kappa^{+3})^{V[G]}=H(\kappa^{+3})[G]$, it suffices to show that
 $H(\kappa^{+3}) \subseteq M^{*}$ and $G \in M^{*}.$

For this purpose we introduce some special functions in $V$.
Let $F: \kappa \rightarrow \kappa$ be defined by $F(\alpha)=\alpha^{+3}.$ Then $j(F)(\kappa)= \kappa^{+3}.$ Now for each $a \in [\kappa^{+3}]^{<\omega}$ with $\kappa \in a$ and $|a|=n$ define the function $G_{a}: [\kappa]^{n} \rightarrow \kappa$ by $G(\alpha_1, ..., \alpha_n)=\alpha_i^{+3}$ where $\kappa$ is the $i-$th element of $a$. It is clear that $j(G_{a})(a)=j(F)(\kappa)= \kappa^{+3}.$
 Also let $r: \kappa \rightarrow H(\kappa)$ be defined by $r(\alpha)=H(\alpha).$

Suppose $f:[\kappa]^{n} \rightarrow H(\kappa)^{V[G]}$ is in $V[G]$ and $a$ is a finite subset of $\kappa^{+3}$ containing $\kappa.$ We say the pair $(f, a)$ has the property $(*)$ iff
\begin{center}
 $\{\gamma: f(\gamma) \in r\circ G_{a}(\gamma) \} \in E_{a}^{*}.$ \footnote{It can be shown that $(f, a)$ has property $(*)$ iff $[f]_{E^*_a}$ represents an element of $H(\kappa^{+3})$ in $M^*_a$.}
\end{center}
We have the following easy lemma.
\begin{lemma}
$(a)$ If $j^{*}(f)(a)=j^{*}(g)(b)$ where $\kappa$ is an element of both $a$ and $b$, then $(f, a)$ has the property $(*)$ iff $(g, b)$ has the property $(*)$,

$(b)$ If $(f, a)$ has the property $(*)$ and $j^{*}(g)(b) \in j^{*}(f)(a)$ for some $b$ containing $\kappa$, then $(g, b)$ has the property $(*).$
\end{lemma}
\begin{lemma}
If $(f, a)$ has the property $(*),$ then there is a function $h: [\kappa]^{m} \rightarrow H(\kappa)$ in $V$ and a finite set $b \subseteq \kappa^{+3}$ such that $j^{*}(f)(a)=j^{*}(h)(b).$
\end{lemma}
\begin{proof}
Let $B = \{\gamma: f(\gamma) \in r\circ G_{a}(\gamma) \}.$ Since $(f, a)$ has the property $(*), B \in E_{a}^{*}.$
 Let $\dot{B}$ be a name for $B$ and let $p \vdash \ulcorner \dot{B} \in \dot{E_{a}^{*}} \urcorner.$ This means that there is some $q \leq^{*} j(p)$ such that $q\upharpoonright \delta=j(p)\upharpoonright\delta=p$ and $q \vdash^{M} \ulcorner a \in j(\dot{B}) \urcorner.$ Hence we have $q \vdash^{M} \ulcorner  j(\dot{f})(a) \in j(r\circ G_{a})(a)=H(\kappa^{+3}) \urcorner.$

For each $c \in H(\kappa^{+3})$ let $\Phi_{c}$ be the sentence ``$j(\dot{f})(a)=c$''. By applying Lemma 2.1.(d) we can find $r \leq^{*} q$ such that for every $c \in H(\kappa^{+3})$
\begin{itemize}
\item $r\upharpoonright\delta=q\upharpoonright\delta=p,$ \item if $s \leq r$ and $s$ decides $\Phi_{c}$ then $(s\upharpoonright\delta)\cup (r\upharpoonright[\delta, j(\kappa)))$ decides $\Phi_{c}$ in the same way.
\end{itemize}
Now $r \vdash^{M} \ulcorner  j(\dot{f})(a) \in j(r\circ G_{a})(a)=H(\kappa^{+3}) \urcorner,$ hence there are $s \leq r$  and $c \in H(\kappa^{+3})$ such that $s \vdash \ulcorner \Phi_{c} \urcorner.$ Let $t=(s\upharpoonright\delta)\cup (r\upharpoonright[\delta, j(\kappa)))$. By above, $t \vdash^{M} \ulcorner \Phi_{c} \urcorner.$

Since $c \in H(\kappa^{+3}),$ there is a function $h:[\kappa]^{m} \rightarrow H(\kappa)$ and a finite $b \subseteq \kappa^{+3}$ such that $c=j(h)(b).$  Thus $t \vdash^{M} \ulcorner j(\dot{f})(a)=j(h)(b) \urcorner$ and the result follows.
\end{proof}

Define the sets $X$ and $X^{*}$ as follows
\begin{center}

$X=\{j(f)(a): (f, a)$ is in $V$ and has the  property $(*)\},$

$\hspace{.7cm}$ $X^{*}=\{j^{*}(f)(a): (f, a)$ is in $V[G]$ and has the  property $(*) \}.$
\end{center}
It follows from Lemma 2.5 that $X$ and $X^{*}$ are transitive.
\begin{lemma}
If $(f, a)$ has the property $(*)$ and $f \in V,$ then $j^{*}(f)(a)=j(f)(a).$
\end{lemma}
\begin{proof}
Define $\Phi: X \rightarrow X^*$ by  $\Phi(j(f)(a))=j^{*}(f)(a).$ Then:
\\
(1) $\Phi$ is well-defined: To see this suppose that $j(f)(a)=j(g)(b).$ We may further suppose that $a=b.$ It then follows that  $j(f)(a)=k_{a}([f]_{E_a})=k_{a}([g]_{E_a})=j(g)(b),$ and hence $B=\{x : f(x)=g(x)    \} \in E_a.$ By Lemma 2.3(b), $B \in E_{a}^*$ and hence $j^{*}(f)(a)=k^{*}_{a}([f]_{E^{*}_a})=k^{*}_{a}([g]_{E^{*}_a})=j^{*}(g)(b).$
\\
(2) $\Phi$ preserves the $\in$ relation: As in (1).

Thus $\Phi$ is an isomorphism, and since both of $X$ and $X^{*}$ are transitive, it must be the identity. The lemma follows.
\end{proof}
\begin{lemma}
$H(\kappa^{+3}) \subseteq M^{*}.$
\end{lemma}
\begin{proof}
We have $H(\kappa^{+3}) \subseteq X \subseteq X^{*} \subseteq M^{*}.$
$\Box$
\begin{lemma}
$G \in M^{*}$
\end{lemma}
\begin{proof}
First note that $\mathbb{P}$$_{S} \in H(\kappa^{+3}) \subseteq M^{*}.$ Define $f: \kappa \rightarrow H(\kappa)^{V[G]}$ by $f(\alpha)=G_{\alpha},$ where $G_{\alpha}= G \cap H(\alpha)$ is $\mathbb{P}$$_{S} \cap H(\alpha)-$generic over $V$.  Show that $G=j^{*}(f)(\kappa),$ and hence $G \in M^{*}.$ By maximality of $G$ it suffices to show that $G \subseteq j^{*}(f)(\kappa).$

Let $p \in G.$ Choose $h:[\kappa]^{n} \rightarrow H(\kappa)$ in $V$ and a finite set $a \subseteq \kappa^{+3}$ containing $\kappa$ such that $p=j(h)(a).$ Then by Lemma 2.7 $p=j^{*}(h)(a).$ Define $f_{a}(\alpha_1, ..., \alpha_n)=f(\alpha_i),$ where $\kappa$ is the $i-$th element of $a$. Then $j^{*}(f_{a})(a)=j^{*}(f)(\kappa).$  Now we have to prove that $j^{*}(h)(a) \in j^{*}(f_{a})(a).$

Let $\dot{f_{a}}$ be a $\MPB_{S}-$name for $f_{a}$ such that $ \vdash_{\MPB_S} \ulcorner \dot{f_{a}}(\alpha_1, ..., \alpha_n)=\dot{G_{\alpha_{i}}}   \urcorner.$ Then $\vdash_{j(\MPB_S)} \ulcorner j(\dot{f_{a}})(a)=\dot{G} \urcorner$ and hence
$\vdash_{j(\MPB_S)} \ulcorner j(h)(a) \in j(\dot{f_{a}})(a) \urcorner.$ The lemma follows.
\end{proof}

\subsection{Coding}

Friedman \cite{glc} presents a method for creating reals which are
class-generic (but not set-generic) over a sufficiently $L$-like
model, preserving Woodin cardinals. A similar method can be used to
preserve strong cardinals. However the general problem of coding a predicate
into a real while preserving large cardinal properties is open;
 we show here that this is possible if the predicate is a
sequence which is generic for a discrete Prikry product.

\begin{theorem}
Suppose that $K$ is the canonical inner model for an $H(\kappa^{+3})$-strong
cardinal $\kappa$. Suppose that $S$ is the discrete set consisting of those measurable
cardinals less than $\kappa$ in $K$ which are not limits of measurable
cardinals in $K$. Also let $(x_\alpha:\alpha\in S)$ be
$\mathbb{P}$$_S$-generic over $K$ for the measures $(U_\alpha:\alpha\in
S)$, where $U_\alpha$ is the unique normal measure on $\alpha$ in $K$. Then there
is a cofinality-preserving set-forcing $\mathbb{P}$ for adding a real $R$
over $K[(x_\alpha:\alpha\in S)]$ such that $K[(x_\alpha:\alpha\in S)][R] = K[R]$ and $\kappa$ remains
$H(\kappa^{+3})$-strong in $K[R]$.
\end{theorem}

\begin{proof} We will follow the proof of Jensen's coding theorem from
\cite{my-book}, section 4.2, making use of Lemma \ref{fuchs} to argue
that the relevant $\Sigma_1$ Skolem hulls taken with respect to
certain initial segments
of $K$ are also $\Sigma_1$ elementary when the Prikry product generic
is adjoined. We must impose some minor changes to the notion of
``string $s$'' and to the coding structures $\mathcal{A}^{s}$, $\tilde{\mathcal{A}}^{s}$,
but for the most part the argument remains the same. The preservation
of $H(\kappa^{+3})$-strength is based on ideas from \cite{glc}.

We work in $L[E][(x_\alpha:\alpha\in S)]$ where
$K=L[E]$ is a fine-structural inner model built from the sequence $E$ of
(partial) extenders. Abbreviate $(x_\alpha:\alpha\in S)$ as $\vec x$ and for
any $\beta$ let $\vec x(\le\beta)$ denote
$(x_\alpha:\alpha\in S$, $\alpha\le\beta)$. We may also assume that
for $\alpha$ in $S$, the min of $x_\alpha$ is greater than the
supremum of $S\cap \alpha$, using the discreteness of the set $S$.
Let $A$ denote the union of the $x_\alpha$, $\alpha \in S$.

$\Card$ denotes the class of infinite cardinals. For $\alpha$ in
$\Card$ we define the ordinals $\mu^{<\eta},\mu^\eta$ by induction
on $\eta\in[\alpha,\alpha^+)$. An ordinal $\mu$ is a \emph{$\ZF^-$ ordinal}
iff $L_\mu[E,\vec x(\le\alpha)]$ is a model of $\ZF$ minus Power
Set. Define:
$\mu^{<\eta}=\cup\{\mu^\xi: \xi<\eta\}\cup\alpha$, $\mu^\eta=$ the
least limit of $\ZF^-$ ordinals $\mu$ such that $\mu$ is
greater than $\mu^{<\eta}$ and, setting
$\mathcal{A}^\eta=L_\mu[E,\vec x(\le\alpha)]$ we have that
$\mathcal{A}^\eta\models\alpha$ is the largest cardinal.

$S_\alpha$, the set of \emph{strings} at $\alpha$ consists of all
$s:[\alpha,|s|)\to 2$, $\alpha\le |s|<\alpha^+$, such that $|s|$
is a multiple of $\alpha$ and $s$ belongs to ${\mathcal A}^{|s|}$.
We write $s\le t$ when $t$ extends $s$ and $s<t$ when $t$ properly
extends $s$.  For $s\in S_\alpha$ we write $\mathcal{A}^s$ for
$\mathcal{A}^{|s|}$ and $\mu^s$ for $\mu^{|s|}$.

For later use (see ``Limit Precoding'') we also define
$\tilde\mu^s < \mu^s$ to be the
least $\ZF^-$ ordinal $\mu$ greater than $\mu^{<|s|}$ such that the
structure $L_\mu[E,\vec x(\le\alpha)]$ contains $s$ and satisfies that
$\alpha$ is the largest cardinal. The resulting structure $\tilde{\mathcal
  A}^s=L_{\tilde\mu^s}[E,x(\le\alpha)]$ is a proper initial segment of
${\mathcal A}^s$ and, like ${\mathcal A}^s$, each element of $\tilde{\mathcal A}^s$ is
$\Sigma_1$ definable in $\tilde{\mathcal A}^s$ from parameters in $\alpha\cup\{\vec
x(\le\alpha),s\}$. (We say that $\tilde{\mathcal A}^s$, ${\mathcal A}^s$
are \emph{$\Sigma_1$ projectible to $\alpha$} with parameters $\vec x(\le\alpha),s$.)

To set up the coding we need the functions $f^s$, defined as follows:
For $\alpha$ an uncountable cardinal, $s$ in $S_\alpha$ and
$i<\alpha$ let $H^s(i)$ denote the $\Sigma_1$ Skolem hull of
$i\cup\{\vec x(\le\alpha),s \}$ in ${\mathcal A}^s$. Then $f^s(i)$ is the
ordertype of $H^s(i)\cap Ord$. For $\alpha$ a successor cardinal we
define the coding set $b^s$ to be the range of $f^s\upharpoonright B^s$
where $B^s$ consists of the successor elements of $\{i<\alpha: i$
is a limit of $j$ such that $j=H^s(j)\cap\alpha\}$.

We describe a cofinality-preserving forcing which codes $K[\vec x]$
into $K[X]$ for some $X\subseteq\omega_1$, preserving the
$H(\kappa^{+3})$-strength of $\kappa$. Then a simple $c.c.c$ forcing can
be used to code $X$ into the desired real $R$.

We need a partition of the ordinals into four pieces: Let $B,C,D,F$
denote the classes of ordinals which are congruent to $0,1,2,3$ mod
$4$, respectively (The letters $A$ and $E$ are already used for other purposes). For any ordinal $\alpha$, $\alpha^B$ denotes the
$\alpha$-th element of $B$ and for any set $Y$ of ordinals, $Y^B$
denotes the set of $\alpha^B$ for $\alpha$ in $Y$ (similarly for
$C,D,F$).

\emph{The successor coding:} Suppose $\alpha\in\Card$ and $s\in
S_{\alpha^+}$. A condition in $R^s$ is a pair $(t,t^*)$ where $t\in
S_\alpha$, $t^*\subseteq \{b^{s\upharpoonright\eta}: \eta\in
[\alpha^+,|s|)\}\cup |t|$, $card(t^*)\le\alpha$.
Extension is defined by: $(t_0,t_0^*)\le (t_1,t_1^*)$ iff $t_0$
extends $t_1$, $t_0^*$ contains $t_1^*$ and:
\\
(1) If $|t_1|\le\gamma^B <|t_0|$ and $\gamma\in
b^{s\upharpoonright\eta}\in t_1^*$ then
$t_0(\gamma^B)=0$ or $s(\eta)$.\\
(2) If $|t_1|\le\gamma^C < |t_0|$ and $\gamma=\langle
\gamma_0,\gamma_1\rangle$ with $\gamma_0\in A \cap t_1^*$ then
$t_0(\gamma^C)=0$ (where $\langle\cdot,\cdot\rangle$ is G\"odel
pairing of ordinals).

An $R^s$-generic over ${\mathcal A}^s$ adds (and is uniquely determined
by) a function $T:\alpha^+\to 2$ such that $s(\eta)=0$ iff
$T(\gamma^B)=0$ for sufficiently large $\gamma\in
B^{s\upharpoonright\eta}$ and such that for $\gamma_0<\alpha^+$,
$\gamma_0\in A$ iff $T(\langle \gamma_0,\gamma_1\rangle^C)=0$ for
sufficiently large $\gamma_1<\alpha^+$.

\emph{The limit precoding.}
Suppose that $\alpha$ is an infinite cardinal and $s$ belongs to
$S_\alpha$.
We say that $X\subseteq\alpha$ \emph{precodes} $s$
if $X$ is the $\Sigma_1$ theory of $\tilde{\mathcal A}^s$ with parameters
from $\alpha\cup\{\vec x(\le\alpha),s\}$, viewed as a subset of $\alpha$.

\emph{The limit coding.}
Suppose that $\alpha$ is an uncountable limit cardinal, $s\in
S_\alpha$ and $p$ is a sequence $((p_\beta,p_\beta^*): \beta\in
\Card\cap\alpha)$ where $p_\beta\in S_\beta$ for each
$\beta\in\Card\cap\alpha$. We will define what it means for $p$ to
``code $s$''. First define the sequence
$(s_\gamma:\gamma\le\gamma_0)$ of elements of $S_\alpha$ as
follows: Let $s_0=\emptyset$. For limit $\gamma\le\gamma_0$,
$s_\gamma$ is the union of the $s_\delta$, $\delta<\gamma$. Now
suppose that $s_\gamma$ is defined and for successor cardinals
$\beta$ less than $\alpha$ let $f_p^{s_\gamma}(\beta)$ be the least
  $\delta\ge f^{s_\gamma}(\beta)$ such that $p_\beta(\delta^D)=1$, if
  such a $\delta$ exists. If $f_p^{s_\gamma}(\beta)$ is undefined for
cofinally many successor cardinals $\beta<\alpha$ then set $\gamma_0
=\gamma$. Otherwise define $X\subseteq\alpha$ by: $\delta\in X$ iff
$p_\beta((f^{s_\gamma}_p(\beta)+1+\delta)^D)=1$ for sufficiently large
successor cardinals $\beta<\alpha$. If $Even(X)=\{\delta:
2\delta\in X\}$ precodes an element $t$ of $S_\alpha$ extending
$s_\gamma$ such that ${\mathcal A}^t$ contains $X$ and the function
$f^{s_\gamma}_p$, then set $s_{\gamma+1}=t$. Otherwise let
$s_{\gamma+1}$ be $s_\gamma * X^F$ (i.e. the concatenation of
$s_\gamma$ with $X^F$ viewed as a sequence of length $\alpha$),
provided $s_\gamma*X^F$ belongs to $S_\alpha$ and $f^{s_\gamma}_p$
belongs to ${\mathcal A}^{s_\gamma*X^F}$; if not, then again set
$\gamma_0=\gamma$. Now $p$ \emph{exactly codes} $s$ if $s$ equals one
of the $s_\gamma$, $\gamma\le\gamma_0$ and $p$ \emph{codes} $s$ is an
initial segment of some $s_\gamma$, $\gamma\le\gamma_0$.

Finally we define the desired forcing. Let $\Card'$ denote the class
of uncountable limit cardinals. Also fix an extender ultrapower embedding
$j:V=K[\vec x]\to M=K^*[\vec x^*]$ witnessing that $\kappa$
is $H(\kappa^{+3})$-strong in $K[\vec x]$.
I.e., $j$ has critical point
$\kappa$, $H(\kappa^{+3})$ of $V$ is contained in $M$ and every
element of $M$ is of the form $j(f)(\alpha)$ for some
$f:\kappa\to V$ in $V$ and $\alpha<\kappa^{+3}$.

\emph{The conditions.} A condition in $\mathbb{P}$ is a sequence
$p=((p_\alpha,p_\alpha^*): \alpha\in \Card$, $\alpha\le\alpha(p))$
where $\alpha(p)\le\kappa^{+3}$ in $\Card$ and:
\\
(1) $p_{\alpha(p)}$ belongs to $S_{\alpha(p)}$ and
$p_{\alpha(p)}^*=\emptyset$.\\
(2) For $\alpha\in\Card\cap\alpha(p)$, $(p_\alpha,p_\alpha^*)$ belongs
to $R^{p_{\alpha^+}}$.\\
(3) For $\alpha\in\Card'$, $\alpha\le\alpha(p)$,
$p\upharpoonright\alpha$ belongs to ${\mathcal A}^{p_\alpha}$ and exactly
codes $p_\alpha$.\\
(4) For $\alpha\in\Card'$, $\alpha\le\alpha(p)$, if $\alpha$ is
inaccessible in ${\mathcal A}^{p_\alpha}$ then there exists a closed
unbounded subset $C$ of $\alpha$, $C\in {\mathcal A}^{p_\alpha}$, such
that for $\beta\in C$, $p^*_\beta=p^*_{\beta^+}=p^*_{\beta^{++}}=
p_{\beta^+}=p_{\beta^{++}}=\emptyset$.

Conditions are ordered by: $p\leq q$ iff:\\
(a) $\alpha(p)\geq \alpha(q)$.\\
(b) $p(\alpha)\le q(\alpha)$ in $R^{p_{\alpha^+}}$ for
$\alpha\in\Card\cap\alpha(p)\cap(\alpha(q)+1)$.\\
(c) $p_{\alpha(p)}$ extends $q_{\alpha(q)}$ if
$\alpha(p)=\alpha(q)$.\\
(d) If $\alpha(q)\geq \kappa^{++}$,
$|q_{\kappa^{++}}|\leq \gamma< |p_{\kappa^{++}}|$,
$\xi< |j(q)_{\kappa^{+3}}|$ is of the
form $j(f)(i)$ for some $i<|q_{\kappa^{++}}|$ and function $f$ with
domain $\kappa$,
$j(q)_{\kappa^{+3}}(\xi)=0$
and
$\gamma$
belongs to $b^{j(q)_{\kappa^{+3}}\upharpoonright\xi}$ (as defined in
$K^*[\vec x^*]$, the ultrapower of $K[\vec x]$ by $j$)
then $p_{\kappa^{++}}(\gamma^B)=0$.

Clause (d) is to ensure that $G_{\kappa^{++}}$, the subset of
$\kappa^{+3}$ added by the generic $G$, codes the union of the
$j(p)_{\kappa^{+3}}$ for $p$ in $G$, a fact needed for the
preservation of $H(\kappa^{+3})$-strength (see below).

This completes the definition of $\mathbb{P}$. The verification of cofinality
and $GCH$ preservation for $\mathbb{P}$ is as in \cite{my-book}, section 4.2,
following the proofs of the Lemmas 4.3 -- 4.6 found there.
Here we only point out the added points to be made, taking into
account that we are coding $\vec x$ over $K=L[E]$ and not over $L$.
For this verification, requirement (4) above can be weakened to only
require that $p^*_\beta=\emptyset$ for $\beta \in C$; the stronger form
of (4) above is needed for the preservation of $H(\kappa^{+3})$-strength.

A general fact that is needed throughout the proof is the following.

\begin{lemma} (Condensation)
Suppose that $\alpha$ is an uncountable cardinal, $s\in S_\alpha$,
$i<\alpha$ and as before let $H^s(i)$ denote the $\Sigma_1$ Skolem
hull of $i\cup\{\vec x(\le\alpha),s\}$ in ${\mathcal A}^s$.\\
(a) If $\alpha$ is a successor cardinal then for sufficiently large
$i<\alpha$, if $i$ is a limit point of $\{j<\alpha: j=H^s(j)\cap
\alpha \}$ then the transitive collapse of $H^s(i)$ is of the form $\bar
K[\vec {\bar x}]$ where $\bar K$ is an initial segment of $K$.\\
(b) If $\alpha$ is a limit cardinal then for sufficiently large
cardinals $i<\alpha$ the transitive collapse of $H^s(i)$ is of the form $\bar
K[\vec {\bar x}]$ where $\bar K$ is an initial segment of $K$.\\
The same holds with ${\mathcal A}^s$ replaced by any of its initial
segments which contain $s$ and have height equal to a $\ZF^-$ ordinal.
\end{lemma}

\begin{proof}
Recall that $s$ belongs to ${\mathcal A}^s=L_{\mu^{|s|}}[E,\vec
x(\leq\alpha)]$. Now $x(\leq\alpha)$ is generic over $K$ for the
product ${\mathbb P}_{S(\leq\alpha)}$
of Prikry forcings at $\beta\leq \alpha$ in $S$. If $\alpha$
is in the closure of $S$ then the intersection of ${\mathbb
  P}_{S(\leq\alpha)}$ with $L_\mu[E]$ is a class forcing in
$L_\mu[E]$ whenever $\mu$ is a $\ZF^-$ ordinal of size $\alpha$ such that
$\alpha$ is the largest cardinal in $L_\mu[E]$.
Nevertheless, all definable antichains in this
forcing are sets. An examination of the proof of Lemma \ref{fuchs} in
\cite{fuchs} reveals that any sequence which satisfies the geometric
property of that lemma with respect to $L_\mu[E]$ for the forcing
${\mathbb P}_{S(\leq\alpha)}\cap L_\mu[E]$ is in fact generic for this
forcing over $L_\mu[E]$. It follows that $x(\leq\alpha)$, which
satisfies the geometric property with respect to the entire $L[E]$, is generic
over $L_\mu[E]$ for this forcing. From this we infer the $\Sigma_1$
definability of the forcing relation for $\Delta_0$ formulas for the
forcing ${\mathbb P}_{S(\leq\alpha)}\cap L_{\mu^s}[E]$ and therefore
that for $i\leq\alpha$, $H^s_0(i)=$ the $\Sigma_1$ Skolem hull of $i\cup\{\dot s\}$ in
${\mathcal A}^s_0$ ($=L_{\mu^{|s|}}[E]$) is equal to the intersection with
${\mathcal A}^s_0$ of $H^s(i)=$ the $\Sigma_1$ Skolem hull of $i\cup\{s\}$ in
${\mathcal A}^s$ (where $\dot s$ is a name for $s\in{\mathcal A}^s$).
In particular, setting $i$ equal to $\alpha$, we see that ${\mathcal A}^s_0$
is $\Sigma_1$-projectible to $\alpha$ with parameter $\dot s$.

If $i$
satisfies the requirements stated in (a) or (b) above, then the
$\Sigma_1$ projectum of the transitive collapse of $H^s_0(i)$ is equal
to $i$ and if $i$ is sufficiently large, then this transitive collapse
is also sound. It follows that $\bar K=$ the transitive collapse of $H^s_0(i)$
is an initial segment of $K$ for such $i$. The last statement of the
lemma follows by the same argument, as any initial segment of ${\mathcal
  A}^s$ which contains $s$ is $\Sigma_1$ projectible to $\alpha$ with
parameter $s$.
\end{proof}

Using Condensation as above, the proofs of Lemmas 4.3 -- 4.6 from
\cite{my-book}, section 4.2 can be carried out in the present
setting:

In Lemma 4.3, one must take the $\alpha_i$'s to enumerate the
first $\alpha$ sufficiently large elements of $\{\beta <\alpha^+:
\beta$ is a limit of $\bar\beta$ such that $\bar\beta=\alpha^+\cap
\Sigma_1$ Skolem hull of $(\bar\beta\cup\{x\})$ in $\mathcal A\}$ which
are sufficiently large so that Condensation (a) guarantees that the
transitive collapse of the associated $\Sigma_1$ hull is of the form
$\bar K[\vec{\bar x}]$ with $\bar K$ an initial segment of $K$. This
facilitates the proof of the Claim in the proof of Lemma 4.3

In Lemma 4.4 one applies Condensation (b) to ensure that the $\Sigma_1$
Skolem hull $H_\beta$, when $\beta=\alpha\cap H_\beta$, transitively
collapses to a structure built from an initial segment of $K$ for
sufficiently large cardinals $\beta<\alpha$; this is needed to argue
that the resulting $s_\beta$ is a string at $\beta$. The rest of the
proof remains unchanged.

The proof of Lemma 4.5 (a) in the case of $\beta$ inaccessible also
uses Condensation (b) in the proof of the Claim, to verify that the
$p^\lambda_\gamma$ are strings (in $S_\gamma$). Also note that
Jensen's subtle use of the assumption that $0^\#$ does not exist
(referred to in the Note) has no counterpart here, as our
structures ${\mathcal A}_0^s=L_{\mu^s}[E]$, $s\in S_\alpha$ collapse $|s|$ to
$\alpha$ without the use of $s$ as an additional predicate (indeed,
$s$ is just a parameter in $L_{\mu^s}[E,\vec x(\le\alpha)]$). The
proofs of Lemma 4.5 in the case of singular $\beta$ as well as Lemma
4.6 can be carried out as before.

We are left with the verification that $\kappa$ remains
$H(\kappa^{+3})$-strong after forcing with $\mathbb{P}$. Recall that $j:V=K[\vec x]\to
M=K^*[\vec x^*]$ is the extender ultrapower embedding witnessing that
$\kappa$ is $H(\kappa^{+3})$-strong.
Let $G$ be $\mathbb{P}-$generic over $V$; in $V[G]$ we must produce a $G^M$ which is
$j(\mathbb{P})$-generic over $M$ and which contains $j(p)$ for each $p$ in $G$.

If $(D_i:i < \kappa)$ are dense subsets of $\mathbb{P}$ and $p$ belongs to
$\mathbb{P}$ then $p$ has an extension $q$ which ``reduces each $D_i$ below
$i^{+3}$'', i.e., any extension $r$ of $q$ can be further extended to
meet $D_i$ without changing $r(\beta)$ for $\beta\ge i^{+3}$. (This
is a variant of $\Delta$-distributivity, see page 30 of
\cite{my-book}.) From this it follows that if we take the upward
closure of $j[G]$, we obtain a compatible set of conditions which
reduces each dense subset of $j(\mathbb{P})$ in $M$ below $\kappa^{+3}$, using
the ultrapower representation of $M$. Moreover, thanks to requirement
(4) in the definition of $\mathbb{P}$, $j[G]$ contains no nontrivial information
between $\kappa$ and $\kappa^{+3}$ (except for $G_\kappa$, the subset
of $\kappa^+$ added by $G$), and therefore $j[G]$ is compatible with
$G\cap H(\kappa^{+3})$. Moreover, thanks to condition (d) in the
definition of extension of conditions, $G_{\kappa^{++}}$ will code the
union of the $j(p)_{\kappa^{+3}}$, $p\in G$, and this coding is
generic (using the fact that the $j(p)_{\kappa^{+3}}$ belong to ${\mathcal A}^\emptyset$;
see Lemma 4.8 of \cite{my-book}).
So we can take $G^M$ to be generated by the joins of
conditions in $j[G]$ with those in $G\cap H(\kappa^{+3})$ to obtain the desired
$j(\mathbb{P})$-generic over $M$.
\end{proof}

\subsection{Killing the $GCH$ everywhere by a cardinal preserving forcing}

In \cite{merimovich} the following is proved.
\begin{theorem} (Merimovich \cite{merimovich})\label{merimovich} Suppose that $GCH$ holds and $\kappa$ is $H(\kappa^{+4})-$
strong. Then there exists a  generic extension of the universe in which  $\kappa$ remains inaccessible and $\forall \lambda \leq \kappa, 2^{\lambda}=\lambda^{+3}.$
\end{theorem}

Unfortunately in the Merimovich model a lot of cardinals are collapsed below $\kappa.$ We show that a simple modification of his proof can give us the the total failure of the $GCH$ below $\kappa$ without collapsing any cardinals.
\begin{theorem} Suppose that $GCH$ holds and $\kappa$ is $H(\kappa^{+3})-$
strong. Then there exists a  cardinal preserving generic extension of the universe in which  $\kappa$ remains inaccessible and $\forall \lambda \leq \kappa, 2^{\lambda} > \lambda^+.$
\end{theorem}
\begin{remark}
In fact it suffices to have a Mitchell increasing sequence of extenders of length $\kappa^{+}$, each of them $(\kappa+2)-$strong.
Thus the exact strength that we need  is a measurable cardinal $\kappa$ with
$o(\kappa)=\kappa^{++}+\kappa^{+}.$
\end{remark}

The idea behind the proof is simple. We consider Merimovich's proof of
Theorem \ref{merimovich} and  replace the collapsing functions introduced in his
proof by suitable Cohen forcings for adding many new sets. We also
need to replace the Cohen forcings used in the proof of Theorem \ref{merimovich}
by new ones because of our weaker assumption. As a result we will get
a model in which we have $2^\lambda = \lambda^{++}$ for a club of
cardinals $\lambda$ below $\kappa$.
As requested by the referee, we now provide the details.

\emph{Extender Sequences}

Suppose $j: V^{*} \rightarrow M^{*} \supseteq V_{\lambda}^{*},
crit(j)=\kappa.$ Define an extender sequence (with projections)

\begin{center}
$E(0)= \langle \langle E_{\alpha}(0): \alpha \in \emph{A} \rangle,
\langle \pi_{\beta, \alpha}: \beta, \alpha \in \emph{A}, \beta
\geq_{j} \alpha \rangle \rangle$
\end{center}

on $\kappa$ by:

\begin{itemize}
\item $\emph{A}=[\kappa, \lambda),$ \item $\forall \alpha \in \emph{A}, E_{\alpha}(0)$ is the $\kappa-$complete ultrafilter on $\kappa$ defined by

    \begin{center}
    $X \in E_{\alpha}(0) \Leftrightarrow \alpha \in j(X)$
    \end{center}
We write $E_\alpha(0)$ as $U_\alpha$.
    \item $\forall \alpha, \beta \in \emph{A}$
     \begin{center}
     $\beta \geq_{j} \alpha \Leftrightarrow \beta \geq \alpha$ and for some $ f \in$$ ^{\kappa} \kappa,$ $  j(f)(\beta)=\alpha$
     \end{center}
     \item $\beta \geq_{j} \alpha \Rightarrow \pi_{\beta, \alpha}: \kappa \rightarrow \kappa$ is such that $j(\pi_{\beta, \alpha})(\beta)=\alpha$

\end{itemize}

For the basic properties $E(0)$ we refer to \cite{gitik} where it is called ``nice system'' there.

Now suppose that we have defined the sequence $\langle E(\tau'): \tau' < \tau  \rangle$. If $\langle E(\tau'): \tau' < \tau  \rangle \notin M^*$ we stop the construction and set

\begin{center}
$\forall \alpha \in \emph{A}, \bar{E}_{\alpha}= \langle \alpha, E(0), ..., E(\tau'), ...: \tau'< \tau \rangle$
\end{center}
and call $\bar{E}_{\alpha}$ an extender sequence of length $\tau$  $(\len(
\Es_{\alpha})=\tau).$

If $\langle E(\tau'): \tau' < \tau  \rangle \in M^*$  then we define
an extender sequence (with projections)

\begin{center}
$E(\tau)= \langle \langle E_{\langle \alpha, E(\tau'): \tau'< \tau \rangle}(\tau): \alpha \in \emph{A} \rangle, \langle \pi_{\langle \beta, E(\tau'): \tau'< \tau \rangle, \langle \alpha, E(\tau'): \tau'< \tau \rangle}: \beta, \alpha \in \emph{A}, \beta \geq_{j} \alpha \rangle \rangle$
\end{center}
on $V_{\kappa}$ by:

\begin{itemize}
\item $X \in E_{\langle \alpha, E(\tau'): \tau'< \tau \rangle}(\tau) \Leftrightarrow \langle \alpha, E(\tau'): \tau'< \tau \rangle \in j(X),$
\item for $\beta \geq_{j} \alpha$ in $\emph{A}, \pi_{\langle \beta, E(\tau'): \tau'< \tau \rangle, \langle \alpha, E(\tau'): \tau'< \tau \rangle}(\langle \nu, d \rangle)= \langle \pi_{\beta, \alpha}(\nu), d \rangle $
\end{itemize}

Note that $ E_{\langle \alpha, E(\tau'): \tau'< \tau \rangle}(\tau)$ concentrates on pairs of the form $\langle \nu, d \rangle$ where $\nu < \kappa$ and $d$ is an extender sequence. This makes the above definition well-defined.

We let the construction run until it stops due to the extender
sequence not being in $M^*$.

\begin{definition}
\begin{enumerate}
\item $\bar{\mu}$ is an extender sequence if there are $j: V^{*} \rightarrow M^{*}$ and $\bar{\nu}$ such that $\bar{\nu}$ is an extender sequence derived from $j$ as above (i.e $\bar{\nu}=\bar{E_{\alpha}}$ for some $\alpha$) and $\bar{\mu}=\bar{\nu}\upharpoonright \tau$ for some $\tau \leq \len(\bar{\nu}),$

\item $\kappa(\bar{\mu})$ is the ordinal of the beginning of the sequence (i.e $\kappa(\bar{E}_{\alpha})=\alpha$),

\item $\kappa^{0}(\bar{\mu})=(\kappa(\bar{\mu}))^{0}$ (i.e $\kappa^{0}(\bar{E}_{\alpha})= \kappa)$),

\item The sequence $ \langle \bar{\mu_{1}}, ..., \bar{\mu_{n}} \rangle$ of extender sequences is $^{0}-$increasing if $\kappa^{0}(\bar{\mu_1}) < ... < \kappa^{0}(\bar{\mu_n}),$

\item The extender sequence $\bar{\mu}$ is permitted to a $^{0}-$increasing sequence $ \langle \bar{\mu_{1}}, ..., \bar{\mu_{n}} \rangle$ of extender sequences if $\kappa^{0}(\bar{\mu_n})<\kappa^{0}(\bar{\mu}),$

\item $X \in \bar{E}_{\alpha} \Leftrightarrow \forall \xi < \len(\bar{E}_{\alpha}), X \in E_{\alpha}(\xi),$

\item $\bar{E}= \langle \bar{E}_{\alpha}: \alpha \in A  \rangle$ is an extender sequence system if there is $j: V^{*} \rightarrow M^{*}$ such that each $\bar{E}_{\alpha}$ is derived from $j$ as above and $\forall \alpha, \beta \in \emph{A}, \len(\bar{E}_{\alpha})= \len(\bar{E}_{\beta}).$ Call this common length, the length of $\bar{E}, \len(\bar{E}),$

\item For an extender sequence $\bar{\mu},$ we use $\bar{E}(\bar{\mu})$ for the extender sequence system containing $\bar{\mu}$ (i.e $\bar{E}(\bar{E}_{\alpha})= \bar{E}$),

\item $\dom(\bar{E})=A$,

\item $\bar{E}_{\beta}  \geq_{\bar{E}} \bar{E}_{\alpha} \Leftrightarrow \beta  \geq_{j} \alpha.$
\end{enumerate}
\end{definition}

\emph{Finding generic filters}

  Start with $GCH$ and construct  an extender sequence system $\bar{E}= \langle \bar{E}_{\alpha}: \alpha \in \dom\bar{E} \rangle$ where $\dom\bar{E}=[\kappa, \kappa^{++})$ and $\len(\bar{E})=\kappa^{+}$ such that $j_{\bar{E}}:V^{*}\rightarrow M_{\bar{E}}^{*}\supseteq V_{\kappa^{++}}^{*}.$ We may suppose that $\bar{E}$ is derived from an elementary embedding $j: V^{*} \rightarrow M^{*}.$   Consider the following elementary embeddings $\forall \tau' < \tau < \len(\bar{E})$
\begin{align*} \label{E-system}
& j_\gt\func  \VS \to \MSt \simeq \Ult(\VS, E(\gt)),
\notag \\
&  k_\gt(j_\gt(f)(\Es_\ga \restricted \gt))=
        j(f)(\Es_\ga \restricted \gt),
\\
\notag & i_{\gt', \gt}(j_{\gt'}(f)(\Es_\ga \restricted \gt')) =
    j_\gt(f)(\Es_\ga \restricted \gt'),
\\
\notag & \ordered{\MSE,i_{\gt, \Es}} = \limdir \ordered {
        \ordof{\MSt} {\gt < \len(\Es)},
                \ordof{i_{\gt',\gt}} {\gt' \leq \gt < \len(\Es)}
        }.
\end{align*}
We restrict $\len(\Es)$ by demanding
$\forall \gt < \len(\Es)$
        $\Es \restricted \gt \in \MSt$.

Thus we get the following commutative diagram.

\begin{align*}
\begin{diagram}
\node{\VS}
        \arrow[3]{e,t}{j}
        \arrow{sse,l}{j_{\gt'}}
        \arrow[2]{se,l}{j_\gt}
        \arrow{seee,t,1}{j_{\Es}}
    \node{}
    \node{}
    \node{\MS}
\\
\node{}
    \node{}
    \node{}
        \node{\MSE}
        \arrow{n,r}{k_{\Es}}
\\
\node{}
    \node{M^*_{\gt'}}
        \arrow[2]{ne,t,3}{k_{\gt'}}
        \arrow{nee,t,2}{i_{\gt', \Es}}
        \arrow{e,b}{i_{\gt', \gt}}
    \node{\MSt = \Ult(\VS, E(\gt))}
        \arrow[1]{ne,b}{i_{\gt, \Es}}
        \arrow{nne,b,1}{k_{\gt}}
\end{diagram}
\end{align*}

Note that
\begin{itemize}
 \item the critical point of elementary embeddings originating in $V^*$ is $\kappa,$  \item the critical point of elementary embeddings originating in other models is $\kappa^{+3}$ as computed in that model.
\end{itemize}
Thus we get

\begin{align*}
& \crit i_{\gt',\gt} = \crit k_{\gt'} = \crit i_{\gt',\Es} =
         (\gk^{+3})_{M^*_{\gt'}},
\\
& \crit k_{\gt} = \crit(i_{\gt, \Es}) = (\gk^{+3})_{M^*_{\gt}},
\\
& \crit k_\Es = (\gk^{+3})_{\MSE}.
\end{align*}
Each of these models catches $V_{\gk+2}^{\MS} = \VS_{\gk+2}$ hence compute
$\gk^{++}$ to be the same ordinal in all models. The larger $\gt$
is the more resemblence there is between $\MSt$ and $\MS$, and hence with $\VS$
towards $\VS_{\gk+3}$. This can be observed by noting that
\begin{center}
 $ \kappa^{+3}_{M_{\tau^{'}}^{*}} < j_{\tau^{'}}(\kappa) < \kappa^{+3}_{M_{\tau}^{*}} < j_{\tau}(\kappa) < \kappa^{+3}_{M_{\bar{E}^{*}}} \leq \kappa^{+3}_{M^{*}} \leq \kappa^{+3}.$
\end{center}
We also factor through the normal ultrafilter to get the following commutative diagram

\begin{align*}
\begin{aligned}
\begin{diagram}
\node{\VS}
        \arrow{e,t}{j_\Es}
        \arrow{se,t}{j_\gt}
        \arrow{s,l}{i_U}
        \node{\MSE}
\\
\node{\NS \simeq \Ult(\VS, U)}
         \arrow{e,b}{i_{U, \gt}}
         \arrow{ne,b}{i_{U, \Es}}
        \node{\MSt}
         \arrow{n,b}{i_{\gt, \Es}}
\end{diagram}
\end{aligned}
\begin{aligned}
\qquad
\begin{split}
& U = E_\gk(0),
\\
& i_U \func  \VS \to \NS \simeq \Ult(\VS, U),
\\
& i_{U, \gt}(i_U(f)(\gk)) = j_\gt(f)(\gk),
\\
& i_{U, \Es}(i_U(f)(\gk)) = j_\Es(f)(\gk).
\end{split}
\end{aligned}
\end{align*}
$\NS$ catches $\VS$ only up to $\VS_{\gk+1}$ and we have
\begin{align*}
\gk^+ < \crit i_{U, \gt} = \crit i_{U, \Es}
    = \gk^{++}_{\NS} < i_U(\gk) < \gk^{++}.
\end{align*}

Let $\langle \PP_{\nu}, \lusim{\QQ_{\nu}} : \nu \leq \kappa     \rangle$ be the reverse Easton iteration such that for any $\nu \leq \kappa:$
\begin{itemize}
\item If $\nu$ is accessible then $\vdash_{\nu} \ulcorner \lusim{\QQ_{\nu}}$ is the trivial forcing$\urcorner$,
\item If $\nu$ is inaccessible then $\vdash_{\nu} \ulcorner \lusim{\QQ_{\nu}}=\lusim{\Add}(\nu^+, \nu^{+3})\times \lusim{\Add}(\nu^{++}, \nu^{+4})\times \lusim{\Add}(\nu^{+3}, \nu^{+5}).$
\end{itemize}
Then we can obtain the following lifting diagram for some suitable $\PP_{\kappa}*\lusim{\QQ_{\kappa}}-$generic filter $G_{<\kappa}*H$:

\begin{align*}
\begin{diagram}
\node{V = \VS[G_{\upto \gk}][H]}
		\arrow[2]{e,t}{j_\Es}
		\arrow{s,l}{i_{U}}
		\arrow{se,b}{j_{\gt'}}
		\arrow{see,b}{j_\gt}
	\node{}
        \node{\ME = \MSE[G^\Es][H^\Es]}
\\
	\node{N = \NS[G^U][H^U]}
		 \arrow{e,b}{i_{U, \gt'}}
	\node{M_{\gt'} = M^*_{\gt'}[G^{\gt'}][H^{\gt'}]}
		 \arrow{ne,t,3}{i_{\gt', \Es}}
		 \arrow{e,b}{i_{\gt', \gt}}
	\node{\Mt = \MSt[G^\gt][H^\gt]}
		\arrow[1]{n,b}{i_{\gt, \Es}}
\end{diagram}
\end{align*}

Set

\begin{align*}
& R_U =(\Add(\kappa^{+4}, i_{U}(\kappa^{++})))_{N^*[G_{<\kappa}]},
\\
& R_\tau =(\Add(\kappa^{+4}, i_{\tau}(\kappa^{++})))_{M_{\tau}^*[G_{\tau}]},
\\
& R_{\bar{E}} =(\Add(\kappa^{+4}, i_{\bar{E}}(\kappa^{++})))_{M_{\bar{E}}^*[G_{\bar{E}}]}
\end{align*}

\begin{lemma}
In $V$ there are $I_{U}, I_{\tau}$ and $I_{\bar{E}}$ such that:

$(a)$ $I_{U}$ is $R_{U}-$generic over $N$,

$(b)$ $I_{\tau}$ is $R_{\tau}-$generic over $M_{\tau}$,

$(c)$ $I_{\bar{E}}$ is $R_{\bar{E}}-$generic over $M_{\bar{E}}$,

$(d)$ The generics are so that we have the following lifting diagram

\begin{align*}
\begin{diagram}
\node{}
    \node{}
        \node{\ME[I_\Es]}
\\
    \node{N[I_U]}
         \arrow{e,b}{i^*_{U, \gt'}}
    \node{M_{\gt'}[I_{\gt'}]}
         \arrow{ne,t,3}{i^*_{\gt', \Es}}
         \arrow{e,b}{i^*_{\gt', \gt}}
    \node{\Mt[I_\gt]}
        \arrow[1]{n,b}{i^*_{\gt, \Es}}
\end{diagram}
\end{align*}
\end{lemma}

Let $i^2_U$ be the iterate of $i_U$. We choose a function, $R(-,-)$, such that
\begin{align*}
& R_U = i^2_U(R)(\gk, i_U(\gk)).
\end{align*}

 The following lemma gives us everything that we need about the model $N[I_{U}]$.
\begin{lemma}
$(a)$  $N[I_{U}]$ and $N$ have the same cardinals,

$(b)$ The power function in $N[I_{U}]$ differs from the power function of $N$ at the following point: $ 2^{\kappa^{+4}}=i_{U}(\kappa)^{++}.$
\end{lemma}

Also The following lifting says everything which we can possibly say about the models  $M_{\tau}[I_{\tau}]$ and $M_{\bar{E}}[I_{\bar{E}}]$.

\begin{align*}
\begin{diagram}
\node{}
	\node{}
        \node{\ME[I_\Es]}
\\
	\node{N[I_U]}
		 \arrow{e,b}{i^*_{U, \gt'}}
	\node{M_{\gt'}[I_{\gt'}]}
		 \arrow{ne,t,3}{i^*_{\gt', \Es}}
		 \arrow{e,b}{i^*_{\gt', \gt}}
	\node{\Mt[I_\gt]}
		\arrow[1]{n,b}{i^*_{\gt, \Es}}
\end{diagram}
\end{align*}
The forcing notion  $\PE$ we define later, due to Merimovich, adds a club to $\gk$.
For each $\gn_1, \gn_2$ successive points in the club
the cardinal structure and power function in the range $[\gn_1^+, \gn_2^{++}]$
of the generic extension
is the same as the cardinal structure and power function in the range
$[\gk^+, j_\Es(\gk)^{++}]$ of $\ME[I_\Es]$.

\emph{Redefining extender Sequences}

 We define a new extender sequence system $\bar{F}= \langle \bar{F}_{\alpha}: \alpha \in \dom(\bar{F})\rangle$ by:
\begin{itemize}
  \item $\dom(\bar{F})=\dom(\bar{E}),$ \item $\len(\bar{F})=\len(\bar{E})$ \item $\leq_{\bar{F}}=\leq_{\bar{E}},$ \item $F(0)=E(0),$ \item $I(\tau)=I_{\tau},$ \item $\forall 0< \tau < \len(\bar{F}), F(\tau)= \langle \langle F_{\alpha}(\tau): \alpha \in \dom(\bar{F}) \rangle, \langle \pi_{\beta, \alpha}: \beta, \alpha \in \dom(\bar{F}), \beta \geq_{\bar{F}} \alpha \rangle \rangle$  is such that
    \begin{center}
    $X \in F_{\alpha}(\tau) \Leftrightarrow \langle \alpha, F(0), I(0), ..., F(\tau^{'}), I(\tau^{'}), ...: \tau^{'}  < \tau \rangle \in j_{\bar{E}}(X),$
    \end{center}
and
\begin{center}
$\pi_{\beta, \alpha}(\langle \xi, d \rangle)= \langle \pi_{\beta, \alpha}(\xi), d \rangle, $
\end{center}
\item $\forall \alpha \in \dom(\bar{F}), \bar{F}_{\alpha}= \langle \alpha, F(\tau),I(\tau): \tau < \len(\bar{F})  \rangle.$
\end{itemize}
Also let $I(\bar{F})$ be the filter generated by $\bigcup_{\tau < \len(\bar{F})} i_{\tau, \bar{E}}^{''}I(\tau).$ Then $I(\bar{F})$ is $\mathbb{R}_{\bar{E}}-$generic over $M_{\bar{E}}.$

From now on we work with this new definition of extender sequence
system and use $\bar{E}$ to denote it.

\begin{definition}$(1)$ $T \in \bar{E}_{\alpha} \Leftrightarrow \forall \xi < \len(\bar{E}_{\alpha}), T \in E_{\alpha}(\xi),$

$(2)$ $T \backslash \bar{\nu} = T \backslash V_{\kappa^{0}(\bar{\nu})}^{*},$

$(3)$ $T \upharpoonright \bar{\nu}= T \cap V_{\kappa^{0}(\bar{\nu})}^{*}.$

\end{definition}

\emph{Definition of the forcing notion $\mathbb{P}_{\bar{E}}$}

 This forcing notion is  the forcing notion of \cite{merimovich}. We give it in detail for completeness. First we define a forcing notion $\mathbb{P}_{\bar{E}}^{*}.$

\begin{definition}
A condition $p$ in $\mathbb{P}_{\bar{E}}^{*}$ is of the form
\begin{center}
$p = \{ \langle \bar{\gamma}, p^{\bar{\gamma}}\rangle: \bar{\gamma} \in s \} \cup \{\langle \bar{E}_{\alpha}, T, f, F \rangle  \}$
\end{center}
where
\begin{enumerate}
\item $s \in [\bar{E}]^{\leq \kappa}, \min\bar{E}= \bar{E}_{\kappa} \in s,$

 \item$p^{\Es_{\kappa}} \in V_{\kappa^{0}(\bar{E})}^{*}$ is an extender sequence such that $\kappa(p^{\bar{E}_{\kappa}})$ is inaccessible ( we allow $p^{\bar{E}_{\kappa}}= \emptyset$). Write $p^0$ for $p^{\bar{E}_{\kappa}}.$

\item $\forall \bar{\gamma} \in s \backslash \{ \min(s) \}, p^{\bar{\gamma}} \in [V_{\kappa^{0}(\bar{E})}^{*}] ^{< \omega}$ is a $^{0}$-increasing sequence of extender sequences and $\max\kappa(p^{\bar{\gamma}})$ is inaccessible,

\item $\forall \bar{\gamma} \in s, \kappa(p^{0}) \leq \max\kappa(p^{\bar{\gamma}})$,

\item $\forall \bar{\gamma} \in s, \bar{E}_{\alpha} \geq \bar{\gamma},$

\item $T \in \bar{E}_{\alpha},$

\item $\forall \bar{\nu} \in T, \mid \{ \bar{\gamma} \in s: \bar{\nu}$ is permitted to $p^{\bar{\gamma}} \} \mid \leq \kappa^{0}(\bar{\nu}),$

\item $\forall \bar{\beta}, \bar{\gamma} \in s, \forall \bar{\nu} \in T,$ if $\bar{\beta} \neq \bar{\gamma}$ and $\bar{\nu}$ is permitted to $p^{\bar{\beta}}, p^{\bar{\gamma}},$ then $\pi_{\bar{E}_{\alpha}, \bar{\beta}}(\bar{\nu}) \neq \pi_{\bar{E}_{\alpha}, \bar{\gamma}}(\bar{\nu}),$

\item $f$ is a function such that

$\hspace{.5cm}$ $(9.1)$ $\dom(f)= \{\bar{\nu} \in T: \len(\bar{\nu})=0 \},$

$\hspace{.5cm}$ $(9.2)$ $f(\nu_{1}) \in R(\kappa(p^{0}), \nu_{1}^{0}).$ If $p^{0}=\emptyset,$ then $f(\nu_{1})= \emptyset,$

\item $F$ is a function such that

$\hspace{.5cm}$ $(10.1)$ $\dom(F)= \{ \langle \bar{\nu_{1}}, \bar{\nu_{2}} \rangle \in T^{2}:\len(\bar{\nu_{1}})=\len(\bar{\nu_{2}})= \emptyset \},$

$\hspace{.5cm}$ $(10.2)$ $F(\nu_{1}, \nu_{2}) \in R(\nu_{1}^{0}, \nu_{2}^{0}),$

$\hspace{.5cm}$ $(10.3)$ $j_{\bar{E}}^{2}(\alpha, j_{\bar{E}}(\alpha)) \in I(\bar{E})$.
\end{enumerate}
\end{definition}
We write $\mc(p), \supp(p), T^{p}, f^{p}$ and $F^{p}$ for $\bar{E}_{\alpha}, s, T, f$ and $F$ respectively.
\begin{definition}
For $p, q \in \mathbb{P}_{\bar{E}}^{*},$ we say $p$ is a Prikry extension of $q$ ($p \leq^{*} q$ or $p \leq^{0} q$) iff
\begin{enumerate}
\item $\supp(p) \supseteq \supp(q),$

\item  $\forall \bar{\gamma} \in \supp(q), p^{\bar{\gamma}}=q^{\bar{\gamma}},$

\item  $\mc(p) \geq_{\bar{E}} \mc(q),$

\item  $\mc(p) >_{\bar{E}} \mc(q) \Rightarrow \mc(q) \in \supp(p),$

\item  $\forall \bar{\gamma} \in \supp(p) \backslash \supp(q), \max\kappa^{0}(p^{\bar{\gamma}}) > \bigcup \bigcup j_{\bar{E}}(f^{q})(\kappa(\mc(q))),$

\item  $T^{p} \leq \pi_{\mc(p), \mc(q)}^{-1''}T^{q},$

\item  $\forall \bar{\gamma} \in \supp(q), \forall \bar{\nu} \in T^{p},$ if $\bar{\nu}$ is permitted to $p^{\bar{\gamma}},$ then
\begin{center}
$\pi_{\mc(p), \bar{\gamma}}(\bar{\nu})=\pi_{\mc(q), \bar{\gamma}}(\pi_{\mc(p), \mc(q)}(\bar{\nu})),$
\end{center}

\item  $\forall \nu_{1} \in \dom(f^{p}), f^{p}(\nu_{1}) \leq f^{q}\circ\pi_{\mc(p), \mc(q)}(\nu_{1}),$

\item  $\forall \langle \nu_{1}, \nu_{2} \rangle \in \dom(F^{p}), F^{p}(\nu_{1}, \nu_{2}) \leq F^{q}\circ \pi_{\mc(p), \mc(q)}(\nu_{1}, \nu_{2}).$
\end{enumerate}

\end{definition}

We are now ready to define the forcing notion $\mathbb{P}_{\bar{E}}.$
\begin{definition}
A condition $p$ in $\mathbb{P}_{\bar{E}}$ is of the form
\begin{center}
$p=p_{n} ^{\frown} ...^{\frown} p_{0} $
\end{center}
where
\begin{itemize}
\item $p_{0} \in \mathbb{P}_{\bar{E}}^{*}, \kappa^{0}(p_{0}^{0}) \geq \kappa^{0}(\bar{\mu}_{1}),$ \item $p_{1} \in \mathbb{P}_{\bar{\mu}_{1}}^{*}, \kappa^{0}(p_{1}^{0}) \geq \kappa^{0}(\bar{\mu}_{2}),$

$\vdots$

    \item $p_{n} \in \mathbb{P}_{\bar{\mu}_{n}}^{*}.$
\end{itemize}
and $\langle \bar{\mu}_{n}, ..., \bar{\mu}_{1}, \bar{E} \rangle$ is a $^{0}-$inceasing sequence of extender sequence systems, that is $\kappa^{0}(\bar{\mu}_{n}) < ... < \kappa^{0}(\bar{\mu}_{1}) < \kappa^{0}(\bar{E}).$
\end{definition}
\begin{definition}
For $p, q \in \mathbb{P}_{\bar{E}},$ we say $p$ is a Prikry extension of $q$ ($p \leq^{*} q$ or $p \leq^{0} q$) iff
\begin{center}
$p=p_{n} ^{\frown} ...^{\frown} p_{0} $

$q=q_{n} ^{\frown} ...^{\frown} q_{0} $
\end{center}
where
\begin{itemize}
\item $p_{0}, q_{0} \in \mathbb{P}_{\bar{E}}^{*}, p_{0} \leq^{*} q_{0},$  \item $p_{1}, q_{1} \in \mathbb{P}_{\bar{\mu}_{1}}^{*}, p_{1} \leq^{*} q_{1},$

   $\vdots$

     \item $p_{n}, q_{n} \in \mathbb{P}_{\bar{\mu}_{n}}^{*}, p_{n} \leq^{*} q_{n}.$
\end{itemize}
\end{definition}
Now let $p \in \mathbb{P}_{\bar{E}}$ and $\bar{\nu} \in T^{p}.$ We define $p_{\langle \bar{\nu} \rangle}$ a one element extension of $p$ by $\bar{\nu}.$
\begin{definition}
Let $p \in \mathbb{P}_{\bar{E}}, \bar{\nu} \in T^{p}$ and $\kappa^{0}(\bar{\nu}) > \bigcup \bigcup j_{\bar{E}}(f^{p, \Col})(\kappa(\mc(p)))$, where $f^{p, \Col}$ is the collapsing part of $f^{p}$. Then $p_{\langle \bar{\nu}\rangle}=p_{1} ^{\frown} p_{0}$ where
\begin{enumerate}
\item $\supp(p_{0})=\supp(p),$

\item $\forall \bar{\gamma} \in \supp(p_{0}),$
$p_{0}^{\bar{\gamma}} = \left\{
\begin{array}{l}
      \pi_{\mc(p), \bar{\gamma}}(\bar{\nu}) \hspace{1.65cm} \text{ if } \bar{\nu} \text{ is permitted to } p^{\bar{\gamma}} \text{ and } \len(\bar{\nu}) >0, \\
       \pi_{\mc(p), \bar{\gamma}}(\bar{\nu}) \hspace{1.65cm} \text{ if } \bar{\nu} \text{ is permitted to } p^{\bar{\gamma}}, \len(\bar{\nu})=0 \text{ and } \bar{\gamma}=\bar{E}_{\kappa},
       \\
       p^{\bar{\gamma} \frown} \langle \pi_{\mc(p), \bar{\gamma}}(\bar{\nu}) \rangle \hspace{.7cm} \text{ if } \bar{\nu} \text{ is permitted to } p^{\bar{\gamma}}, \len(\bar{\nu})=0 \text{ and } \bar{\gamma}\neq \bar{E}_{\kappa},
       \\
       p^{\bar{\gamma}} \hspace{2.95cm} \text{ otherwise }.

     \end{array} \right.$

\item $\mc(p_{0})=\mc(p),$

\item $T^{p_{0}}=T^{p} \backslash \bar{\nu},$

\item $\forall \nu_{1} \in T^{p_{0}}, f^{p_{0}}(\nu_{1})=F^{p}(\kappa(\bar{\nu}), \nu_{1}),$

\item $F^{p_{0}}=F^{p},$

\item if $\len(\bar{\nu})>0$ then

$\hspace{.5cm}$ $(7.1)$ $\supp(p_{1})=\{\pi_{\mc(p), \bar{\gamma}}(\bar{\nu}): \bar{\gamma} \in \supp(p)$ and $\bar{\nu}$ is permitted to $p^{\bar{\gamma}}\},$

$\hspace{.5cm}$ $(7.2)$ $p_{1}^{\pi_{\mc(p), \bar{\gamma}}(\bar{\nu})}=p^{\bar{\gamma}},$

$\hspace{.5cm}$ $(7.3)$ $\mc(p_{1})=\bar{\nu},$

$\hspace{.5cm}$ $(7.4)$ $T^{p_{1}}=T^{p} \upharpoonright \bar{\nu},$

$\hspace{.5cm}$ $(7.5)$ $f^{p_{1}}=f^{p} \upharpoonright \bar{\nu},$

$\hspace{.5cm}$ $(7.6)$ $F^{p_{1}}=F^{p} \upharpoonright \bar{\nu},$

\item if $\len(\bar{\nu})=0$ then

$\hspace{.5cm}$ $(8.1)$ $\supp{p_{1}} =\{ \pi_{\mc(p),0}(\bar{\nu}) \},$

$\hspace{.5cm}$ $(8.2)$ $p_{1}^{\pi_{\mc(p),0}(\bar{\nu})}=p^{\bar{E}_{\kappa}},$

$\hspace{.5cm}$ $(8.3)$ $\mc(p_{1})=\bar{\nu}^{0},$

$\hspace{.5cm}$ $(8.4)$ $T^{p_{1}}= \emptyset,$

$\hspace{.5cm}$ $(8.5)$ $f^{p_{1}}=f^{p}(\kappa(\bar{\nu})),$

$\hspace{.5cm}$ $(8.6)$ $F^{p_{1}}= \emptyset.$
\end{enumerate}
\end{definition}
We use $(p_{\langle \bar{\nu} \rangle})_{0}$ and $(p_{\langle \bar{\nu} \rangle})_{1}$ for $p_{0}$ and $p_{1}$ respectively. We also let $p_{\langle \bar{\nu_{1}}, \bar{\nu_{2}} \rangle }= (p_{\langle \bar{\nu}_{1}\rangle})_{1} ^{\frown} (p_{\langle \bar{\nu}_{1} \rangle})_{0 \langle \bar{\nu_{2}} \rangle}$ and so on.

The above definition is the key step in the definition of the forcing relation $\leq.$
\begin{definition}
For $p, q \in \mathbb{P}_{\bar{E}},$ we say $p$ is a $1-$point extension of $q$ ($p \leq^{1} q$) iff
\begin{center}
$p=p_{n+1} ^{\frown} ...^{\frown} p_{0} $

$q=q_{n} ^{\frown} ...^{\frown} q_{0} $
\end{center}
and there is $0 \leq k \leq n$ such that
\begin{itemize}
\item $\forall i < k, p_{i}, q_{i} \in \mathbb{P}_{\bar{\mu}_{i}}^{*}, p_{i} \leq^{*} q_{i},$ \item $\exists \bar{\nu} \in T^{q_{k}}, (p_{k+1}) ^{\frown}p_{k} \leq^{*} (q_{k})_{ \langle \bar{\nu} \rangle}$ \item $\forall i > k, p_{i+1}, q_{i} \in \mathbb{P}_{\bar{\mu}_{i}}^{*}, p_{i+1} \leq^{*} q_{i},$
\end{itemize}
where $\bar{\mu}_{0}=\bar{E}.$
\end{definition}
\begin{definition}
For $p, q \in \mathbb{P}_{\bar{E}},$ we say $p$ is an $n-$point extension of $q$ ($p \leq^{n} q$) iff there are $p^{n}, ..., p^{0}$ such that
\begin{center}
$p=p^{n} \leq^{1} ... \leq^{1} p^{0}=q.$
\end{center}
\end{definition}
\begin{definition}
For $p, q \in \mathbb{P}_{\bar{E}},$ we say $p$ is an extension of $q$ ($p \leq q$) iff there is some $n$ such that $p \leq^{n} q$.
\end{definition}
Suppose that $H$ is $\mathbb{P}_{\bar{E}}-$generic over $V.$
Then all of the results in \cite{merimovich}, except the following, work as well, :

 Replace Claim 10.3 with:

\begin{claim}
Assume $\len(\bar{E})>0$. Then $\vdash_{\PP_{\bar{E}}} \ulcorner 2^{\kappa}=\kappa^{++} \urcorner$.
\end{claim}

 Replace Claim 10.4 with:

\begin{claim}
Assume $\len(\bar{E})>0$. Then $\vdash_{\PP_{\bar{E}}} \ulcorner 2^{\kappa^{+}}=\kappa^{+3}, 2^{\kappa^{++}}=\kappa^{+4}, 2^{\kappa^{+3}}=\kappa^{+5} \urcorner$.
\end{claim}

 Replace Claim 10.6 with the following:

\begin{claim}
 Let $G$ be $\mathbb{P}$$_{\bar{E}}$-generic with
$p = p_l *... * p_k *...* p_0 \in G$
and $\bar{\epsilon}$ be such that
$p_{l..k} \in \mathbb{P}$$_{\bar{\epsilon}}$ and $l(\bar{\epsilon}) = 0$. Let $\nu = \kappa(p_k^0)$.
Then, in $V[G]$, all cardinals in $[\nu^+, \kappa^{0}(\bar{\epsilon})^{++}]$ are preserved and
    $2^{\nu^{+}} = \nu^{+3}$,
    $2^{\nu^{++}} = \nu^{+4}$,
    $2^{\nu^{+3}} = \nu^{+5}$,
    $2^{\nu^{+4}} = \kappa^{0}(\bar{\epsilon})^{++}$.
\end{claim}

Now the proof of our main theorem goes as follows: Let $p^* \in \mathbb{P}$$_{\bar{E}}^{*}$ such that $\kappa(p^{*0})$ is inaccessible
and  $G$ be $\mathbb{P}$$_{\bar{E}}$-generic with $p^* \in G$. Set

$\hspace{4cm}$ $M = \bigcup \{ p_{0}^{\bar{E}_{\kappa}}: p \in G  \},$

$\hspace{4cm}$ $C = \bigcup \{ \kappa(p_{0}^{\bar{E}_{\kappa}}): p \in G  \}.$

Note that $M$ is a Radin generic sequence for the extender sequence
$\bar{E}_{\kappa},$
hence $C \subset \kappa$ is a club. Also the first ordinal in this
club is $\lambda = \kappa(p^{*0})$. We first investigate the range $(\lambda, \kappa)$ in $V[G].$ Note that, by 10.5 from \cite{merimovich}, for $\bar{\epsilon} \in M$ it is enough to
use $\mathbb{P}$$_{\bar{\epsilon}}$ in order to understand $V_{\kappa^0(\bar{\epsilon})}^{V[G]}$.
So let $\mu \in (\lambda, \kappa).$
\begin{itemize}
\item $\mu \in limC:$ Then there is $\bar{\epsilon} \in M$ such that
        $l(\bar{\epsilon}) > 0$ and $\kappa(\bar{\epsilon}) = \mu$. By 10.7 from \cite{merimovich} $\mu$ remains a cardinal and by Claim 2.27, $2^{\mu}=\mu^{++},$ \item $\mu \in C \setminus \lim C$: Then there is $\bar{\epsilon} \in M$ such that
        $l(\bar{\epsilon}) = 0$ and $\kappa(\bar{\epsilon}) = \mu$.
    Let $\mu_2 \in C$ be the $C$-immediate predecessor of $\mu$. By  Claim 2.29 we have all cardinals in $[\mu_{2}^{+}, \mu^{++} ]$ are preserved and
    $2^{\mu_2^{+}} = \mu_2^{+3}$,
    $2^{\mu_2^{++}} = \mu_2^{+4}$,
    $2^{\mu_2^{+3}} = \mu_2^{+5}$,
    $2^{\mu_2^{+4}} = \mu^{++}$. In particular $2^{\mu} \geq \mu^{++}.$
\item $\mu \notin C:$ Then there are $\mu_2$ and $\mu_1$ two successive points in $C$ such that $\mu \in (\mu_2, \mu_1).$ By above, if $\mu \in \{ \mu_2^+, \mu_2^{++}, \mu_2^{+3}\}$ then $2^{\mu}=\mu^{++},$ and if $\mu \in (\mu_2^{+3}, \mu_1)$ then $2^{\mu}\geq\mu_1^{++} > \mu^{+}.$

\end{itemize}
We may  note that the above argument also shows that all cardinals $> \lambda$ are preserved in $V[G],$ and since forcing with $\mathbb{P}$$_{\bar{E}}$ adds no new bounded subsets to $\lambda,$ hence all cardinals are preserved in $V[G].$ Finally let $H$ be $Add(\aleph_{0}, \lambda^{++})_{V[G]}-$generic over $V[G]$. It is then clear that in $V[G][H]$ all cardinals are preserved and that the $GCH$ fails everywhere below (and at) $\kappa.$

\subsection{Proof of Theorem 1.1}

Suppose that $K$ is the canonical inner model for a $H(\kappa^{+3})$-strong
cardinal $\kappa$. Let $S$ be the discrete set of measurable cardinals
below $\kappa$ in $K$ which are not limits of measurable cardinals in
$K$ and for each $\alpha \in S$ let $U_{\alpha}$ be the unique normal
measure on $\alpha$ in $K$. Consider the forcing  $\mathbb{P}$$_{S}$ and let
$(x_\alpha:\alpha\in S)$ be \emph{$\mathbb{P}$$_S$-generic} over
$K$. By Theorem 2.2, $\kappa$ remains $H(\kappa^{+3})-$strong in
$K[(x_\alpha:\alpha\in S)],$ thus we can apply Theorem 2.10 to find a
cofinality-preserving forcing $\mathbb{P}$ which adds a real $R$  over
$K[(x_\alpha:\alpha\in S)]$ such that $K[(x_\alpha:\alpha\in
S)][R]=K[R]$ and $\kappa$ remains $H(\kappa^{+3})-$strong in $K[R].$
By Theorem 2.13 there exists a cardinal-preserving forcing
$\mathbb{Q}$ and a subset $C \subseteq S$, $\mathbb{Q}$-generic over
$K[R]$ such that in $K[R][C]$, $\kappa$ remains inaccessible and for
every $\lambda < \kappa$, $2^{\lambda} >\lambda^{+}$.
We now define a new sequence $(y_\alpha:\alpha\in S)$ by

\begin{center}
$y_\alpha = \left\{
\begin{array}{l}
 x_{\alpha} \hspace{3.46cm} \text{ if } \alpha \in C,\\
      x_{\alpha}- \{min(x_{\alpha})\}  \hspace{1.4cm} \text{ otherwise }.

     \end{array} \right.$
\end{center}
By Lemma 3,  $(y_\alpha:\alpha\in S)$ is $\mathbb{P}$$_{S}-$generic over $K.$ Let $W=V_{\kappa}^{K[(y_\alpha:\alpha\in S)]}$ and $V=W[R]=V_{\kappa}^{K[R][C]}.$ Then the pair $(W,V)$ is as required.
Theorem 1.1 follows.
\end{proof}

\section{Proof of Theorem 1.2}

\subsection{A class version of the Prikry product}

Let $S$ be a  class of measurable cardinals which is discrete.  Fix normal measures $U_{\alpha}$  on $\alpha$ for  $\alpha$ in $S$. We define a class version of the Prikry product as follows.

Conditions in $\mathbb{P}$$_{S}$ are triples $p=(X^p,S^p,H^p)$ such that

\noindent
(1) $X^p$ is a subset of $S$,\\
(2) $S^{p} \in \prod_{\alpha \in X^{p}}[\alpha\setminus sup(S \cap \alpha)]^{< \omega}$,\\
(3) $H^{p} \in \prod_{\alpha \in X^p} U_{\alpha}$,\\
(4) $supp(p)=\{ \alpha:S^{p}(\alpha) \neq \emptyset \}$ is finite,\\
(5) $\forall \alpha \in X^{p}$, $max \ S^{p}(\alpha) < min \ H^{p}(\alpha)$.

Let $p,q \in \mathbb{P}$$_{S}$. Then $p \leq q$ ($p$ is an extension of $q$) iff

\noindent
(1) $X^p \supseteq X^q$,\\
(2) $\forall \alpha \in X^q, S^p(\alpha)$ is an end extension of $S^q(\alpha)$,\\
(3) $\forall \alpha \in X^q, S^p(\alpha)\setminus S^q(\alpha) \subseteq H^q(\alpha)$,\\
(4) $\forall \alpha \in X^q, H^p(\alpha) \subseteq H^q(\alpha)$.

We also define an auxiliary relation $ \leq^*$ on $\mathbb{P}$$_{S}$ as follows.
Let $p,q \in \mathbb{P}$$_{S}$. Then $p \leq^* q$ ($p$ is a direct or Prikry extension of $q$) iff

\noindent
(1) $X^p \supseteq X^q$,\\
(2)  $\forall \alpha \in X^q, S^p(\alpha) = S^q(\alpha)$,\\
(3) $\forall \alpha \in X^q, H^p(\alpha) \subseteq H^q(\alpha)$.

For $p \leq q$ in $\mathbb{P}$$_{S}$ we define the distance function $|p - q|$ to be a function on $X^q$ so that for $\alpha \in X^q, |p - q|(\alpha)=l(S^p(\alpha))- l(S^q(\alpha)).$
Also let $\mathbb{P}$$_{S}\upharpoonright X = \{ p \in \mathbb{P}$$_{S}: X^{p} \subseteq X \}.$
It is clear that for any
$X \subseteq S, \mathbb{P}$$_{S} \simeq (\mathbb{P}$$_{S}\upharpoonright X) \times (\mathbb{P}$$_{S}\upharpoonright S\setminus X)$.

\begin{lemma}
$\MPB_S$ is pretame: Given $p\in\MPB_S$ and a definable sequence $(D_i:i<\alpha)$ of dense classes below $p$ there exist $q\leq p$ and a sequence $(d_i:i<\alpha)\in V$ such that each $d_i\subseteq D_i$ is predense below $q.$
\end{lemma}
\begin{proof} Let $p_0=p$ and let $\delta_0>\alpha, \delta_0 \notin S$ be such that $X^{p_{o}}\subseteq \delta_0.$ By repeatedly thinning the measure one sets above $\delta_0$ we can find $p_1\leq p_0$ and $\delta_1>\delta_0, \delta_1 \notin S$ such that:
\begin{enumerate}
\item $X^{p_{1}}\subseteq \delta_1,$  \item $p_1$ agrees with $p_0$ below $\delta_0,$ \item for any $q\leq p_0, q \in \MPB_S\upharpoonright\delta_0$ and any $i<\alpha$ if $q$ has an extension $r$ meeting $D_i$ which agrees with $q$ below $\delta_0,$ then there is such an $r \in \MPB_S\upharpoonright\delta_1$ whose measure one sets contain those of $p_1.$
\end{enumerate}
Now repeat this $\omega$ times, producing $p_0, p_1, \ldots$  and let
$q$ be the greatest lower bound, obtained in a natural way, so that $q
\leq^* p_n$ for each $n\in \omega$. Also for each $i<\alpha$ set $d_i
=D_i\upharpoonright \delta_{\omega}=\{r\in D_i:X^{r}\subseteq
\delta_{\omega}  \}$, where $\delta_{\omega}=sup_{n<\omega}\delta_n$.
We show that $q$ and the sequence $(d_i:i<\alpha)$ are as required.

Fix $i<\alpha.$ Suppose $r\leq q, r\in D_i.$ Let $n$ be large enough so that $supp(r)\cap\delta_{\omega}\subseteq\delta_n.$ At stage $n+1$ we considered $r\upharpoonright\delta_n$ and saw that it has an extension meeting $D_i$ and agreeing with it below $\delta_n,$ so it must have such an extension whose measure one sets contain those of $p_{n+1}$ and therefore those of $q.$ This extension is compatible with $r$ and therefore $r$ has an extension which meets $d_i,$ as required.
\end{proof}

It follows from \cite{my-book}, Theorem 2.18, and the above Lemma  that the forcing relation is definable. The proof of the following lemma uses ideas from \cite{magidor}.
\begin{lemma}$(\mathbb{P}$$_{S}, \leq, \leq^*)$ has the Prikry property, i.e for each sentence $\phi$ of the forcing language of  $(\mathbb{P}$$_{S}, \leq)$, and any $p \in \mathbb{P}$$_{S}$ there is $q \leq^* p$ which decides $\phi$.
\end{lemma}
\begin{proof}
Suppose $\phi$ is a sentence of the forcing language, $p \in \PP_S$.
Let $p=(X^p,S^p,H^p)$, let $\phi^0$ denote $\neg \phi$ and $\phi^1$
denote
$\phi$.

By reflection and by strengthening $p$ in the sense of $\leq^*$, we may
assume that $X^p=\gamma$, where it is dense in $\PP_S \cap V_\gamma$ to
decide $\phi$.

For $\alpha < \gamma$, let $\mathcal{S}_\alpha$ denote the set of $S^q$
where $q \in \mathbb{P}_{X_p \cap \alpha}$.
For  $s \in \mathcal{S}_\alpha$, set
$F_{s,\alpha}(\delta_1, \hdots, \delta_n)=i$ iff there is $q \leq p$
such
that $X^q=\gamma$,
$S^q \upharpoonright (X^p \setminus \{ \alpha\}) = s$,
$S^q(\alpha)=S^p(\alpha)*(\delta_1, \hdots \delta_n)$ and $q \Vdash
\phi^i$.
Set $F_{s,\alpha}(\delta_1, \hdots, \delta_n)=2$ iff no such $q$ exists.

Let $H(s, \alpha) \subseteq H^p(\alpha)$, $H(s, \alpha) \in U_\alpha$ be
homogeneous for $F_{s,\alpha}$, and let $H(\alpha)= \bigcap_{s \in
\mathcal{S}_\alpha} H(s,\alpha)$.
Then $H(\alpha) \in U_\alpha$ (as $S$ is discrete) and we can set
$q=(X^q,
S^q, H^q)$,
where $X^q=X^p$, $S^q=S^p$ and $H^q(\alpha) =  H(\alpha)$ for $\alpha
\in
X^q$.

It is clear that $q \leq^* p$. We show that there is a $\leq^*$
extension
of $q$ which decides $\phi$. Suppose not.
Let $r \leq q$ be such that $r$ decides $\phi$. Suppose for example that
$r \Vdash \phi$. We may further suppose that $r$ is so that $| r - q |$
is
minimal, and that $X^r=\gamma$. We note that $| r - q |$ is not the
$0$-funtion.

Let $\alpha<\gamma$ be the maximum of $\operatorname{supp}(r)$, and let
$r_0$ be obtained from $r$ by replacing $S^r(\alpha)$ with
$S^p(\alpha)$.
We claim that $r_0$ already decides $\phi$.
For let $w \leq r_0$, such that $w \Vdash \neg \phi$.
Let $n$ denote $|S^{r}(\alpha)|$;
We may assume that $|S^w(\alpha)| \geq n$.
Let $s$ denote $S^{r_0}$ and $\delta_1, \hdots \delta_k$ denote
$S^{w}(\alpha)$.
Then $r$ witnesses that $F_{s,\alpha}$ has constant value $1$ on
$[H(s,\alpha)]^n$.
Moreover, $\{\delta_1, \hdots \delta_n\} \in [H(s,\alpha)]^n$.
So there is $r_1$ such that $r_1 \Vdash \phi$, $S^{r_1}\upharpoonright
(X^p \setminus \{\alpha \}) =s$ and $S^{r_1}(\alpha)=\{\delta_1, \hdots,
\delta_n\}$.
It is easily checked that $S^{r_1}$ and $S^w\upharpoonright\gamma$ are
compatible, so $r_1$ and $w$ are compatible, contradicting that they
decide $\phi$ differently.
Thus, $r_0$ already decides $\phi$, contradicting the minimality of $r$.

\end{proof}

We can now easily show that $\MPB_{S}$  preserves cardinals and the $GCH.$
Also as in the usual Prikry product a  $\PP_{S}-$generic is uniquely determined by a sequence $\langle x_{\alpha}: \alpha \in S \rangle$ where each $x_\alpha$ is an $\omega-$sequence cofinal in $\alpha.$ As before, with a slight abuse of terminology, we say that $\langle x_{\alpha}: \alpha \in S \rangle$ is $\PP_{S}-$generic. The following is an analogue of Lemma 2.1 and its proof is essentially the same.

\begin{lemma}
(a) The sequence $\langle x_{\alpha}: \alpha \in S \rangle$ obeys the following ``geometric property'': if $\langle X_{\alpha}: \alpha \in S \rangle$ is a definable class (in $V$) and  $X_{\alpha} \in U_{\alpha}$ for each $\alpha \in S$ then $\bigcup_{\alpha \in S}x_{\alpha} \setminus X_{\alpha}$ is finite.

(b) Conversely, suppose that $\langle y_\alpha:\alpha\in S \rangle$ is a sequence
(in any outer model of $V$)
satisfying the geometric property stated above. Then
$\langle y_\alpha: \alpha\in S \rangle$ is $\PP_S$-generic over $V$.
\end{lemma}

\subsection{Proof of Theorem 1.2}

Suppose $M$ is a model of $ZFC+GCH+$ there exists a proper class of measurable cardinals. Let $S$ be a discrete class of measurable cardinals and for each $\alpha \in S$ fix a normal measure $U_{\alpha}$ over $\alpha.$ Consider the forcing  $\mathbb{P}$$_{S}$ and let $\langle x_\alpha:\alpha\in S \rangle$ be \emph{$\mathbb{P}$$_S$-generic} over $M$. By Jensen's coding theorem (see \cite{my-book}) there exists a cofinality-preserving forcing $\mathbb{P}$ which adds a real $R$  over $M[\langle x_\alpha:\alpha\in S \rangle]$ such that $M[\langle x_\alpha:\alpha\in S \rangle][R]=L[R].$
In $L[R]$ define the function $F^*: REG \rightarrow CARD$ by
\begin{center}
$F^{*}(\kappa)=\left\{
\begin{array}{l}
       F(\kappa) \hspace{1.2cm} \text{ if } cfF(\kappa) \neq \omega,\\
       F(\kappa)^+ \hspace{1cm} \text{ if }  cfF(\kappa) = \omega.

     \end{array} \right.$
\end{center}

Let $\mathbb{R}$ be the Easton forcing corresponding to $F^*$ for blowing up the power of each regular cardinal $\kappa$ to $F^{*}(\kappa)$  and let $C \subseteq S$ be $\mathbb{R}-$genreric over $L[R].$

We now define a new sequence $\langle y_{\alpha}: \alpha \in S \rangle$ by
\begin{center}
$y_\alpha = \left\{
\begin{array}{l}
 x_{\alpha} \hspace{3.46cm} \text{ if } \alpha \in C,\\
      x_{\alpha}- \{min(x_{\alpha})\}  \hspace{1.4cm} \text{ otherwise }.

     \end{array} \right.$
\end{center}

Using lemma 3.3,  $\langle y_{\alpha}: \alpha \in S \rangle$ is $\mathbb{P}$$_{S}-$generic over $M$.
Let $W=M[ \langle y_{\alpha}: \alpha \in S \rangle] $, and $V=M[\langle y_{\alpha}: \alpha \in S \rangle, R] $. Then the pair $(W,V)$ is as required. This completes the proof of Theorem 1.2.
\end{proof}

\section{A few more results}

The following is proved in \cite{eslami-gol}. We give a proof for completeness.
\begin{lemma}
Suppose that $R$ is a real in $V.$ Then there are two reals $a$ and $b$ such that

$(a)$ $a$ and $b$ are Cohen generic over $V,$

$(b)$ all of the models $V, V[a], V[b]$ and $V[a,b]$ have the same cardinals,

$(c)$ $R \in L[a,b].$
\end{lemma}
\begin{proof} Working in $V$, let $a^*$ be $Add(\omega, 1)-$generic over $V$ and let $b^*$ be $Add(\omega, 1)-$generic over $V[a^{*}],$ where $Add(\omega, 1)$ is the Cohen
forcing for adding a new real. Note that $V[a^*]$ and $V[a^*,
b^*]$ are cardinal preserving generic extensions of $V$. Working
in $V[a^*, b^*]$ let $\langle k_N: N<\omega \rangle$ be an
increasing enumeration of $\{N: a^{*}(N)=0\}$ and let $a=a^*$ and
$b=\{N: b^{*}(N)=a^{*}(N)=1 \} \cup \{k_N:R(N)=1 \}.$
Then clearly $R \in L[\langle k_N: N<\omega \rangle,
b]\subseteq L[a,b]$ as $R=\{N: k_N \in b \}.$

We show that $b$ is $Add(\omega, 1)-$generic over $V$. It suffices
to prove the following

$\hspace{1.5cm}$ For any $(p,q) \in Add(\omega,
1)*\lusim{Add}(\omega, 1)$ and any dense

(*) $\hspace{0.85cm}$ open subset $D \in V$ of $Add(\omega, 1)$ there is
$(\bar{p},\bar{q}) \leq (p,q)$

$\hspace{1.5cm}$ such that $(\bar{p},\bar{q}) \vdash
 \ulcorner \dot{b}$ extends some element of $D \urcorner$.

Let $(p,q)$ and $D$ be as above. By extending one of $p$ or $q$ if
necessary, we can assume that $lh(p)=lh(q)$. Let $\langle k_N: N<M
\rangle$ be an increasing enumeration of $\{N<lh(p): p(N)=0\}.$
Let $s: lh(p) \rightarrow 2$ be such that considered as a subset
of $\omega,$

\begin{center}
$s=\{ N<lh(p): p(N)=q(N)=1 \} \cup \{k_N: N<M, R(N)=1 \}.$
\end{center}

Let $t \in D$ be such that $t \leq s.$
Extend $p,q$ to $\bar{p}, \bar{q}$ of length $lh(t)$
so that for $i$ in the interval $[lh(s),lh(t))$
\begin{itemize}
\item $\bar{p}(i)=1$, \item $\bar{q}(i)=1$ iff $i \in t.$
\end{itemize}

Then
\begin{center}
$t=\{ N<lh(t): \bar{p}(N)=\bar{q}(N)=1 \} \cup \{k_N: N<M, R(N)=1
\}.$
\end{center}
Thus
$(\bar{p},\bar{q}) \vdash \ulcorner \dot{b}$ extends t $\urcorner$
and $(*)$ follows. This completes the proof of the Lemma.
\end{proof}

Using the above Lemma and Theorems 1.1 and 1.2 we can obtain the following.

\begin{theorem}
Assume the consistency of an $H(\kappa^{+3})$-strong cardinal
$\kappa.$ Then there exist a model $W$ of $ZFC$ and two reals $a$ and $b$ such that

$(a)$ The models $W, W[a]$ and $W[b]$ have the same cardinals and satisfy the

$\hspace{.5cm}$ $GCH,$

$(b)$  $GCH$ fails at all infinite cardinals in $W[a,b].$

\end{theorem}

\begin{theorem}
Let $M$ be a model of $ZFC+GCH+$ there exists a proper class of measurable cardinals. In $M$ let $F:REG \longrightarrow CARD$ be an Easton function. Then there exist a cardinal preserving generic extension $W$ of $M$ and two reals $a$ and $b$ such that

$(a)$ The models $W, W[a], W[b]$ and $W[a,b]$ have the
same cardinals,

$(b)$ $W[a]$ and $W[b]$ satisfy $GCH$,

$(c)$ $W[a,b] \models \ulcorner \forall \kappa \in REG, 2^{\kappa} \geq F(\kappa)
\urcorner.$
\end{theorem}

\noindent{\large\bf Acknowledgement}

The authors would like to thank the referee for supplying the correct proof of Lemma 3.2.

{Kurt G\"{o}del Research Center, University of Vienna,

E-mail address: sdf@logic.univie.ac.at}

{Department of Mathematics, Shahid Bahonar University of Kerman,
Kerman-Iran and School of Mathematics, Institute for Research in
Fundamental Sciences (IPM), P.O. Box: 19395-5746, Tehran-Iran.

E-mail address: golshani.m@gmail.com}

\end{document}